\documentclass{ipi}
\usepackage{amsmath,xcolor}
\usepackage{paralist,amsfonts}
\usepackage{graphics}
\usepackage{epsfig}
\usepackage{graphicx}
\usepackage{epstopdf}
 \usepackage[colorlinks=true]{hyperref}
\hypersetup{urlcolor=blue, citecolor=red}

\def\currentvolume{12}
 \def\currentissue{3}
  \def\currentyear{2018}
   
    \def\ppages{X--XX}
     \def\DOI{10.3934/ipi.2018034}

\newtheorem{theorem}{Theorem}[section]
\newtheorem{corollary}{Corollary}

\newtheorem{lemma}[theorem]{Lemma}
\newtheorem{proposition}{Proposition}

\theoremstyle{definition}
\newtheorem{definition}[theorem]{Definition}
\newtheorem{remark}{Remark}

\title[Inverse Problem For the Magnetic Schr\"{o}dinger Operator]
{An inverse problem for the magnetic Schr\"{o}dinger operator on Riemannian manifolds from partial boundary data}

\author[Sombuddha Bhattacharyya]{}

\subjclass{Primary: 58F15, 31B20;  Secondary: 35J25.}
 \keywords{Differential geometry, inverse problems, Calder\'{o}n problem, geometric analysis, partial differential equations.}

 \email{arkatifr@gmail.com}

\newcommand{\B}[1]{B^{(#1)}}
\newcommand{\q}[1]{q^{(#1)}}
\newcommand{\p}[1]{{#1}^{\prime}}
\newcommand{\e}{\epsilon}
\newcommand{\R}{\mathbb{R}^{n}_{+}}
\newcommand{\D}{\text{ d}}
\newcommand{\g}{\mathfrak{g}}

\begin{document}
\maketitle

\centerline{\scshape Sombuddha Bhattacharyya}
\medskip
{\footnotesize
 \centerline{Tata Institute of Fundamental Research}
  \centerline{Centre for Applicable Mathematics}
   \centerline{Bangalore, India}

} 

\bigskip

 \centerline{(Communicated by Hao-Min Zhou)}

\begin{abstract}
We consider the inverse problem of recovering the magnetic and potential term of a magnetic Schr\"{o}dinger operator on certain compact Riemannian manifolds with boundary from partial Dirichlet and Neumann data on suitable subsets of the  boundary. The uniqueness proof relies  on proving a suitable Carleman estimate for functions which vanish only on a part of boundary and constructing complex geometric optics solutions which vanish on a part of the boundary.
\end{abstract}


\section{Introduction and statement of the main result}
In this article, we consider a Calder\'{o}n type inverse problem involving the magnetic Schr\"odinger operator on a compact Riemannian manifold  with boundary with available Dirichlet and Neumann measurements on suitable subsets of the boundary.
To define the type of domain we are interested in let us first define the notion of a simple manifold.
\begin{definition}
A hypersurface is \textbf{strictly convex} if the second fundamental form is positive definite.
\end{definition}

\begin{definition}[Simple Manifold]
A Riemannian manifold $(M,g)$ with boundary is \textbf{simple} if $\partial M$ is strictly convex, and for any point $x\in M$ the exponential map $\mbox{exp}_{x}$ is a diffeomorphism from its domain in $T_{x}M$ onto $M$.
\end{definition}

The Riemannian manifold that we consider is of the following type.
\begin{definition}[Admissible manifold \cite{F_K_S_U,K_S}]\label{amissible_definition}
We say a compact Riemannian manifold with boundary $(M,g)$ is  \textbf{admissible} if
	\begin{enumerate}
		\item	$n =$ dim($M$) $\geq 3$
		\item	$M$ is orientable
		\item	$(M,g)$ is conformal to a sub manifold (with boundary) of $\mathbb{R} \times (M_0, g_0)$ where $(M_0,g_0)$ is a compact, simple $(n-1)$ dimensional Riemannian manifold.	
	\end{enumerate}
\end{definition}

Let  $(M,g)$ to be an admissible Riemannian manifold and let $B = (B_1, B_2, \dots , B_n)$ be a smooth complex valued 1-form on $M$ and $q \in L^{\infty}(M)$ be a complex valued function.

We define the magnetic Schr\"{o}dinger operator on $M$ as
\begin{equation} \label{mag_schrd_operator}
\mathcal{L}_{B,q} := d_{\overline{B}}^* d_B + q,
\end{equation}
where ${d_B} = d + iB\wedge : C^{\infty}(M) \rightarrow \Omega^1(M)$ and $d_{{B}}^*$ is the formal adjoint
of $d_B$ (for the sesquilinear inner product induced by the Hodge dual on
the exterior form algebra). Here $\Omega^1(M)$ denotes the collection of all 1-forms on $M$.

In local coordinates
\begin{equation}\label{Magnetic Schrodinger local coordinates}
d_{\overline{B}}^*d_B u = -\lvert g \rvert^{-1/2} (\partial_{x_k} + iB_k)(\lvert g \rvert^{1/2}g^{jk}(\partial_{x_j}+ iB_j)u),
\end{equation}
where $|g|=\det(g)$ and the sum is over the indices which are repeating. Through out this article we assume this summation convention that repeated indices are implicitly summed over.
Simplifying Equation \eqref{Magnetic Schrodinger local coordinates}, we get,
\begin{equation}\label{step 1}
\begin{aligned}
\mathcal{L}_{B,q} =	- \Delta_g u - 2i\langle B,du \rangle_g	+ \tilde{q}u,	
\end{aligned}
\end{equation}
where $\tilde{q} = q - i[\lvert g \rvert^{-1/2}\partial_{x_k}(\lvert g \rvert^{1/2}g^{jk}B_j)] + \lvert B \rvert^2_g$.

We assume throughout that $0$ is not an eigenvalue of $\mathcal{L}_{B,q}$ in $M$ and consider the Dirichlet problem
\begin{equation} \label{mag_schrd_eqn}
\begin{aligned}
\mathcal{L}_{B,q} u &= 0 \qquad \mbox{in}\quad M\\
u&=f \qquad \mbox{on} \quad \partial M.
\end{aligned}
\end{equation}

We define the Dirichlet to Neumann map $\Lambda_{B,q}$ as follows:
\begin{align*}
\Lambda_{B,q} : f \rightarrow d_B u(\nu)|_{\partial M}, \quad  f\in H^{1/2}(\partial M),
\end{align*}
where in local coordinates
\begin{align*}
d_B u(\nu) = \nu_j g^{jk} (\partial_k + iB_k)u|_{\partial M}.
\end{align*}

Our goal is to recover the coefficients $B$ and $q$ in $\Omega$ from the knowledge of $\Lambda_{B,q}(f)$ measured on a part of boundary and with $f$ supported on a different part of boundary.

Now we define the subsets of the boundary where we have the boundary information.
Write $x=(x_1,x')$ for points in $\mathbb{R}\times M_0$, where $x_1$ is the Euclidean coordinate.

The function $\phi(x) = x_1$ allows us to decompose the boundary $\partial M$ as the disjoint union
\begin{equation*}
\partial M = \partial_{-} M \cup \partial_{+}M,
\end{equation*}
where
\begin{equation*}
\begin{aligned}
\partial_{-}M &= \{ x\in \partial M : \partial_\nu \phi \leq 0 \}, \\
\partial_{+}M &= \{ x\in \partial M :\partial_\nu \phi \geq 0 \}.
\end{aligned}
\end{equation*}
Here $\partial_\nu \phi$ is with respect to the metric $g$.

Define $\Gamma_D$ and $\Gamma_N$ open subsets of $\partial M$ so that
\begin{equation*}
\Gamma_{N} \supset \partial_{-} M  \quad \mbox{and } \quad
\Gamma_D \supset \partial_{+}M.
\end{equation*}
Hence, $\Gamma_D \cup \Gamma_N = \partial M$.

We define the boundary data as
\begin{equation}\label{Cauchy data}
\mathcal{C}^{\Gamma_D,\Gamma_N}_{B,q} (M) = \{(u|_{\Gamma_D}, d_Bu(\nu)|_{\Gamma_N}):\mathcal{L}_{B,q}u=0,\ \mbox{supp}(u|_{\partial M}) \subset \Gamma_D \},
\end{equation}
where supp$(.)$ denotes support.

\begin{lemma}[Gauge invariance] \label{Gauge_inv}
Let $B,q$ be as above and $\Phi \in C^2(M)$ be such that $\Phi|_{\partial \Omega} = 0$, then we have
\begin{equation*}
\Lambda_{B+d\Phi,q} = \Lambda_{B,q} \qquad on \quad \partial \Omega.
\end{equation*}
\end{lemma}
\begin{proof}
The proof follows from a straight forward calculation (see \cite{DOS}).
\end{proof}
The above lemma shows that one can recover $B$ only up to a term of the form $d \Phi = (\partial_{x_i}\Phi)\,dx^i$ with $\Phi=0$ on $\partial M$ from the boundary data \eqref{Cauchy data}.
Now we now state the main result of the article. Before that we mention that if $v$ is an 1-form defined on $M$, then $dv$ is the 2-form defined as
\begin{equation}\label{ext_der}
(dv)_{ij} = \frac{1}{2}\left( \frac{\partial v_i}{\partial x_j} - \frac{\partial v_j}{\partial x_i} \right).
\end{equation}

\begin{theorem}\label{main_theorem}
Let $(M,g)$ be simply connected and admissible,  $\B{1}$ and $\B{2}$ be two smooth complex valued 1-forms in $M$ with $\B{1}=\B{2}$ on $\partial M$ and let $\q{1}, \q{2}$ be two complex valued $L^{\infty}$ functions on $M$ such that $0$ is not an eigenvalue of $\mathcal{L}_{\B{j},\q{j}}$ for $j=1,2$.
If
\begin{equation*}
\mathcal{C}^{\Gamma_D,\Gamma_N}_{\B{1},\q{1}}(M) = \mathcal{C}^{\Gamma_D,\Gamma_N}_{\B{2},\q{2}}(M),
\end{equation*}
then $d\B{1} = d\B{2}$ and $\q{1}=\q{2}$ in $M$.
\end{theorem}

Inverse problems of the kind considered in this paper has attracted considerable attention in recent years. Calder\'{o}n initiated the study of such inverse problems and  in his original work \cite{Calderon_Paper} investigated the question of unique recovery of conductivity $\gamma$ of a medium $\Omega$ from steady state voltage and current measurements made on the boundary.  In mathematical terms, the question posed by Calder\'on involves the unique recovery of the positive coefficient $\gamma\in L^{\infty}$ in the boundary value problem  $\nabla\cdot \gamma\nabla u=0, u|_{\partial \Omega}=f$  from the boundary data, $f\to \gamma \frac{\partial u}{\partial \nu}|_{\partial \Omega}$, where $\nu$ is the unit outer normal to $\partial \Omega$. Calder\'on was able to establish  the uniqueness result for conductivities close to a constant. The global uniqueness result for $C^{2}$ conductivities was proved by Sylvester and Uhlmann in their fundamental work  \cite{SYL}, where they recast the inverse problem for the conductivity equation to an inverse problem involving the Schr\"{o}dinger equation $(-\Delta + q)u=0$ and introduced the important notion of complex geometric optics solutions for this equation.

In \cite{SUN}, Sun considered the magnetic Schr\"{o}dinger equation in a Euclidean set up and showed that from  the Dirichlet to Neumann map on the full boundary, one can uniquely determine $dB$ (where $d$ denotes the exterior derivative \eqref{ext_der}) and $q$ on a bounded subset $\Omega$ assuming $B$ is small. Here we should note that one cannot recover $B$ completely from the Dirichlet to Neumann map on the boundary; see Lemma \ref{Gauge_inv}. Later Nakamura, Sun and Uhlmann in \cite{S_N_U} removed the smallness assumption and proved that one can uniquely determine $dB$ and $q$ from the boundary data where $B \in C^{\infty}(\Omega)$ and $q\in L^{\infty}(\Omega)$. Additionally there are several works that have improved the regularity condition on the coefficients; see \cite{TOL,SAL_2006,KIM2,KNU,ISA,SAL}

In the case of domains with dimension 2 significant amount of work has already been done. Some of the major works in this direction are \cite{NA,SYL,Astala,BUK,Iman1,Iman2,Iman3}.

In the direction of results concerning Calder\'on type inverse problems with partial boundary data, in dimensions $\geq 3$ Bukhgeim and Uhlmann in \cite{B_U} showed the uniqueness result for the Schr\"{o}dinger equation assuming that the Neumann data is measured on slightly more than half of the boundary. This result was substantially improved by Kenig, Sj\"{o}strand and Uhlmann in \cite{KEN} who showed that unique recovery of the potential function $q$ is possible from boundary measurements on possibly small subsets of the boundary. The analogous result in the setting of magnetic Schrödinger equation was done by Dos Santos Ferreira, Kenig, Sj\"{o}strand and Uhlmann in \cite{DOS}, where they  showed that unique recovery of $dB$ and $q$ is possible from Neumann measurements measured on possibly small subsets of the boundary and with no restriction imposed on the support of the Dirichlet data.

Another natural extension of the Calder\'on inverse problem is to consider the same problem in the setting of a compact  Riemannian manifold with boundary. On admissible Riemannian manifolds, Dos Santos Ferreira, Kenig, Salo and Uhlmann in \cite{F_K_S_U} showed that from full boundary Dirichlet to Neumann data, one can recover $dB$ and $q$ uniquely.
Recently Kenig and Salo \cite{K_S} again in the setting of admissible Riemannian manifolds showed that for the case when $B\equiv 0$ one can restrict both the Dirichlet and the Neumann data on certain subsets of boundary and still recover $q$ uniquely. Furthermore, they also showed that one can ignore a part of boundary while considering the boundary data and can also relax the assumption on the Riemannian metric in the sense that it only needs to be conformally flat in one direction. Very recently in \cite{Krupchyk2017} Krupchyk and Uhlmann showed that on an admissible manifold, one can relax the regularity assumptions on the coefficients of an magnetic Schr\"{o}dinger operator and can still recover the lower order perturbations from the boundary data.

Returning to the magnetic Schr\"{o}dinger equation in the Euclidean setting,  Chung very recently in \cite{chung_2014} proved that one can uniquely recover both $dB$ and $q$ from partial Dirichlet and partial Neumann boundary data.
See also  \cite{chung_2,chung_3} for related results.

Our work extends the results in \cite{chung_2014} and \cite{K_S} since we consider the Magnetic Schr\"{o}dinger inverse problem on an admissible Riemannian manifold and we are interested in the recovery of both $dB$ and $q$ from partial Dirichlet and partial Neumann data. To the best of our knowledge, such a problem has not been considered in previous studies. Due to the method of proof, the boundary sets in our work, are strictly dependent on the direction in which the domain is conformally flat  and hence we can not take arbitrarily small sets for the boundary measurements.

The paper is organized as follows. In Section \ref{sec_car_est} we prove a suitable boundary Carleman estimate. Then in Section \ref{int_Car_est}, following the ideas of \cite{chung_2014}, we will use the $H^1$ interior Carleman estimate in \cite{DOS} to derive an $H^{-1}$ estimate for functions vanishing only on a part of boundary. Using the estimates we construct suitable complex geometric optics type solutions in Section \ref{CGO}.
Next in Section \ref{int_id} we derive an integral identity involving the magnetic and potential terms using the boundary Carleman estimate and construct suitable complex geometric optics type solutions for Equation \ref{mag_schrd_eqn} that are $0$ on a prescribed part of the boundary. Here we will use the interior Carleman estimate to prove the existence of the solution in our desired form. The construction closely follows the construction given in \cite{KEN,F_K_S_U, K_S}. Finally in Section \ref{agrt} we will obtain integral equations involving  $B$ and $q$ and recover $dB$ and $q$ based on unique recovery results involving the attenuated geodesic ray transform \cite{F_K_S_U,K_S,VS}.

\section{Boundary Carleman estimate}\label{sec_car_est}
In this section we prove a Carleman estimate with boundary terms, as in \cite{K_S}, for the conjugated operator $e^{\phi/h}(-\Delta_g)e^{-\phi/h}$ in $M$, where $\phi = \pm x_1$ and $h>0$ small.
In \cite{F_K_S_U} it is shown that on an admissible manifold,  one can consider $\phi(x) = \pm x_1$ to be a limiting Carleman weight for semiclassical Laplacian on $M$.
We refer to \cite{F_K_S_U, HOR} for the definition and properties of limiting Carleman weights on manifolds.
Following  \cite{KEN} and \cite{K_S} we consider a slightly modified weight
\begin{equation*}
\phi_{\e} = \phi + \frac{h\phi^2}{2\e}
\end{equation*}
where $0<h\leq h_0$.
First we prove a small lemma which will allow us to ignore the conformal factor $c$ in further calculations of the Carleman estimate.
\begin{lemma}
Let $g(x) = c(x)\tilde{g}(x)$ where $c(x) \in C^2(M)$ is a positive function and
\begin{equation*}
\mathcal{L}_{B,q} = -\Delta_g - 2i \langle B, d_{g} \rangle_g + q_1,
\end{equation*}
where $q_1 = q - i[\lvert g \rvert^{-1/2}\partial_{x_k}(\lvert g \rvert^{1/2}g^{jk}B_j)] + \lvert B \rvert^2_g$,
then
\begin{equation*}
c(x)\mathcal{L}_{B,q} = -\Delta_{\tilde{g}} - 2i \langle \tilde{B}, d_{\tilde{g}} \rangle_{\tilde{g}} + \tilde{q},
\end{equation*}
where
\begin{gather*}
\tilde{B} = B + ic^{-1}\left(\frac{n}{4} - \frac{1}{2}\right)d_{\tilde{g}}c\\
\tilde{q} = c(x)q_1 = c(x)\left[ q-\frac{ic^{-1}}{\sqrt{\lvert\tilde{g}\rvert}} \frac{\partial}{\partial x_k}\left( \sqrt{\lvert\tilde{g}\rvert} \tilde{g}^{jk} B_j \right) -i\frac{n-2}{2c^2}\langle B, dc\rangle_{\tilde{g}} + c^{-1}\lvert B \rvert^2_{\tilde{g}}\right].
\end{gather*}
\end{lemma}
\begin{proof}

Observe that
\begin{equation*}
\begin{aligned}
\Delta_g u = \frac{1}{\sqrt{\lvert g\rvert}}\frac{\partial}{\partial x^i}\left(\sqrt{\lvert g\rvert} g^{ij} \frac{\partial u}{\partial x^j} \right)
= c^{-1}\left[\Delta_{\tilde{g}} u
+ c^{-1}\left(\frac{n}{2} - 1\right)\langle d_{\tilde{g}}c, d_{\tilde{g}}u \rangle_{\tilde{g}} \right].
\end{aligned}
\end{equation*}

Hence,
\begin{equation*}
\begin{aligned}
\mathcal{L}_{B,q}u &= -\Delta_g u - 2i \langle B, d_{\tilde{g}}u \rangle_g + q_1u \\
&= -c^{-1}(x)\Delta_{\tilde{g}} u - c^{-2}(x)\left(\frac{n}{2} - 1\right)\langle d_{\tilde{g}} c, d_{\tilde{g}} u \rangle_{\tilde{g}}
- 2ic^{-1} \langle B, d_{\tilde{g}}u \rangle_{\tilde{g}} + q_1u
\end{aligned}
\end{equation*}

Now consider
\begin{equation*}
\begin{aligned}
\tilde{B} &= B + ic^{-1}\left(\frac{n}{4} - \frac{1}{2}\right)d_{\tilde{g}}c\\
\tilde{q} &= c(x)q_1(x) \\
&= c\left[ q-ic^{-1}\frac{1}{\sqrt{\lvert\tilde{g}\rvert}} \frac{\partial}{\partial x_k}\left( \sqrt{\lvert\tilde{g}\rvert} \tilde{g}^{jk} B_j \right) -ic^{-2}\left(\frac{n}{2} - 1\right)\langle B, d_{\tilde{g}} c\rangle_{\tilde{g}} + c^{-1}\lvert B \rvert^2_{\tilde{g}}\right].
\end{aligned}
\end{equation*}

Then we will get
\begin{equation*}
c(x)\mathcal{L}_{B,q} = -\Delta_{\tilde{g}} - 2i \langle \tilde{B}, d_{\tilde{g}} \rangle_{\tilde{g}} + \tilde{q}.
\end{equation*}
\end{proof}

A simple calculation shows that $\mathcal{L}_{\tilde{B},q_2} = -\Delta_{\tilde{g}} - 2i \langle \tilde{B}, d_{\tilde{g}} \rangle_{\tilde{g}} + \tilde{q}$, where
\begin{equation*}
q_2 = cq + c^{-1}\left(\frac{n-2}{4}\right)\Delta_{\tilde{g}}c - c^{-2}\left(\frac{n-2}{4}\right)^2\lvert d_{\tilde{g}}c\rvert^2_{\tilde{g}},
\end{equation*}
which implies $c(x)\mathcal{L}_{B,q} = \mathcal{L}_{\tilde{B},q_2}$.

We will see that the Carleman estimate depends on the principal part (highest order term) of the operator $\mathcal{L}$ and therefore by the above lemma we can take $c\equiv 1$ in the calculations of the Carleman estimate. That is taking different $c$ will change the lower order terms which we will later prove that can be absorbed in to the Carleman estimate.

We define the semiclassical Fourier transform on $\mathbb{R}^{n}$ as follows
\begin{equation*}
\hat{u}(\xi) = \frac{1}{(2\pi h)^{n/2}} \int_{\mathbb{R}^n} e^{-i\frac{x}{h} \cdot \xi} u(x) dx.
\end{equation*}

The semiclassical Sobolev spaces $H^s_{scl}(\mathbb{R}^n)$ are defined as
\begin{equation*}
H^s_{scl}(\mathbb{R}^n) = \{ u\in L^2(\mathbb{R}^n)\ :\ (1+\lvert\xi\rvert^{2})^{s/2}\hat{u}(\xi) \in L^2(\mathbb{R}^n) \}.
\end{equation*}

For bounded domain $\Omega$ we can write
\begin{equation*}
\lVert u \rVert_{H^1_{scl}(\Omega)} = \lVert u \rVert_{L^2(\Omega)} + \lVert h\nabla u \rVert_{L^2(\Omega)}.
\end{equation*}
From here onward, unless otherwise specified, we will use the Fourier transform and the Sobolev spaces in semiclassical sense only.
Let us denote
\begin{equation*}
\begin{aligned}
\mathcal{L}_{\phi_{\e}} = -e^{\phi_{\e}/h}h^2\Delta_g e^{-\phi_{\e}/h} \qquad \mbox{and} \quad
\mathcal{L}_{\phi_{\e},B,q} = e^{\phi_{\e}/h}h^2\mathcal{L}_{B,q} e^{-\phi_{\e}/h}.
\end{aligned}
\end{equation*}

Now we will prove a boundary Carleman estimate for $\Delta_g$ on $M$. The main idea of the proof follows \cite[Proposition 4.1]{K_S} but the weight we consider is slightly different from the weight in \cite{K_S}. Hence, for sake of completeness we present the proof of it.

\begin{proposition} \label{initial carleman estimate_boundary}
Let $(M,g)$ be as above, let $\phi = \pm x_1$. Then for some positive constants $h_0$, $C$, $\e$ with $h_0 < \frac{\e}{2}<1$ then for $h<h_0$ and for all $u\in H^1(M)$ with $u|_{\partial M} = 0$ one has,
\begin{equation} \label{c_E_0}
C\lVert \mathcal{L}_{\phi_{\e}} u \rVert^2_{L^2(M)}
	\geq 	\frac{h^2}{\e} \left(\lVert u \rVert^2_{L^2(M)}
	+ 		 \lVert h\nabla u \rVert^2_{L^2(M)}\right)
	- 		2h^3 \langle (\partial_\nu(\phi_\e)\partial_\nu u), \partial_\nu u  \rangle_{L^2(\partial M)}.
\end{equation}
\end{proposition}
\begin{proof}
We have $\mathcal{L}_{\phi_{\e}} = X+iY$ where
\begin{equation*}
X = -h^2\Delta_g - \lvert \nabla\phi_\e \rvert^2_{g}, \qquad Y = -2i\langle \nabla\phi_\e, h\nabla . \rangle_g - ih\Delta_g\phi_\e.
\end{equation*}

Observe that $X$ and $Y$ are self adjoint operators.
For $u \in C^{\infty} (M)$ with $u|_{\partial M} = 0$ we have
\begin{equation*}
\begin{aligned}
\lVert \mathcal{L}_{\phi_{\e}} u\rVert^2_{L^2(M)}
=& \langle (X+iY)u , (X+iY)u \rangle_{L^2(M)}\\
=& \lVert Xu \rVert^2_{L^2(M)} + \lVert Yu \rVert^2_{L^2(M)} -i\langle Xu, Yu \rangle_{L^2(M)} + i\langle Yu, Xu \rangle_{L^2(M)}\\
=& \lVert Xu \rVert^2_{L^2(M)} + \lVert Yu \rVert^2_{L^2(M)} -i\langle YXu,u\rangle_{L^2(M)} + i\langle Yu, Xu \rangle_{L^2(M)}\\
=& \lVert Xu \rVert^2_{L^2(M)} + \lVert Yu \rVert^2_{L^2(M)} -i\langle YXu,u\rangle_{L^2(M)} + i\langle XYu,u\rangle_{L^2(M)} \\
&- ih^2\langle Yu, \partial_\nu u\rangle_{L^2(\partial M)}\\
=& \lVert Xu \rVert^2_{L^2(M)} + \lVert Yu \rVert^2_{L^2(M)} -i\langle YXu,u\rangle_{L^2(M)} + i\langle XYu,u\rangle_{L^2(M)} \\
&- 2h^3\langle (\partial_\nu \phi_\e)\partial_\nu u, \partial_\nu u\rangle_{L^2(\partial M)}
\end{aligned}
\end{equation*}
From the calculation in theorem 4.1 of \cite{K_S} we get
\begin{equation*}
i[X,Y] = \frac{4h^2}{\e}\left(1+ \frac{h}{\e}\phi \right)^2 + hY\beta Y + h^2R,
\end{equation*}
where $\beta = \left[\frac{h}{\e}\left(1+\frac{h}{\e}\phi\right)^{-2}\right] $, $R$ is a first order semiclassical differential operator having coefficients uniformly bounded in $h$ and $\e$, for $h << \e$.
Hence,
\begin{equation*}
i\langle [X,Y] u , u\rangle_{L^2(M)} = \frac{h^2}{\e}\left\lVert \left(1+\frac{h}{\e}\phi\right)u \right\rVert^2_{L^2(M)} + h\langle Y\beta Yu , u \rangle_{L^2(M)} + h^2\langle Ru,u \rangle_{L^2(M)}
\end{equation*}
One can make $h_0$ small enough so that $h\lvert \frac{\phi}{\e}\rvert \leq \frac{1}{2}$ in $M$ for all $h<h_0$. Hence by integration by parts, we get,
\begin{equation*}
h\left\lvert \left\langle Y\left[\frac{h}{\e}\left(1+\frac{h}{\e}\phi\right)^{-2}\right] Yu , u \right\rangle \right\rvert_g
\leq C\frac{h^2}{\e}\lVert Yu \rVert^2_{L^2(M)}.
\end{equation*}
And similarly $h^2\lvert\langle Ru,u \rangle_{L^2(M)}\rvert \leq Ch^2\lVert u\rVert_{H^1(M)}$.

Hence we have
\begin{equation*}
\left\lvert i\langle [X,Y] u , u\rangle_{L^2(M)} \right\rvert_g
\geq \frac{h^2}{\e}\lVert u \rVert^2_{L^2(M)} - C\frac{h^2}{\e}\lVert Y u \rVert^2_{L^2(M)} - Ch^2\lVert u \rVert^2_{H^1(M)}
\end{equation*}
Now as $u|_{\partial M} = 0$, using integration by parts and Young's inequality we get
\begin{equation*}
\begin{aligned}
h^2\lVert h\nabla_g u \rVert^2_{L^2(M)} 	
	&= h^2\langle -h^2\Delta_g u,u \rangle_{L^2(M)} \\
	&= h^2\langle Xu, u\rangle_{L^2(M)} + h^2\langle\lvert\nabla {\phi_{\e}} \rvert^2_g u, u\rangle_{L^2(M)}\\
	&\leq \frac{1}{2K}\lVert Xu \rVert^2_{L^2(M)} + \frac{Kh^4}{2}\lVert u \rVert^2_{L^2(M)} + C_3 h^2\lVert u \rVert^2_{L^2(M)},
\end{aligned}
\end{equation*}
where $K$ is a positive constant whose value will be specified later. Putting all the estimates together we get,
\begin{equation*}
\begin{aligned}
\lVert \mathcal{L}_{\phi_{\e}} u \rVert^2_{L^2(M)} \geq &\left[2Kh^2 \lVert h\nabla u \rVert^2_{L^2(M)} - K^2 h^4\lVert u \rVert^2_{L^2(M)} - 2KC_3h^2\lVert u \rVert^2_{L^2(M)}\right] \\
& + \lVert Yu\rVert^2_{L^2(M)} + \left[\frac{h^2}{\e}\lVert u \rVert^2_{L^2(M)} - C_1\frac{h^2}{\e}\lVert Yu \rVert^2_{L^2(M)} - C_2h^2 \lVert u \rVert^2_{H^1} \right]\\
& - 2h^3 \langle (\partial_{\nu} \phi_{\e}) \partial_\nu u, \partial_\nu u \rangle_{L^2(\partial M)}.
\end{aligned}
\end{equation*}

Now let us choose $h_0$ small enough so that $C_1\frac{h_0^2}{\e} \leq \frac{3}{4}$ and $K = \frac{1}{\alpha \e}$, where $\alpha$ is to be determined. Then for $h\leq h_0$ the above estimate takes the form
\begin{equation*}
\begin{aligned}
\lVert &\mathcal{L}_{\phi_{\e}} u \rVert^2_{L^2(M)} \\
\geq & \frac{h^2}{\e}\left( \lVert u \rVert^2_{L^2(M)} + \frac{2}{\alpha} \lVert h\nabla u\rVert^2_{L^2(M)} \right) - C_2h^2\left( \lVert u \rVert^2_{L^2(M)} + \lVert h\nabla u\rVert^2_{L^2(M)} \right)\\
 & - \frac{h^2}{\e} \left(\frac{h^2}{\e \alpha^2}- 2\frac{C_3}{\alpha}\right)\lVert u \rVert^2_{L^2(M)} + \frac{1}{4}\lVert Yu \rVert^2_{L^2(M)}
- 2h^3 \langle (\partial_{\nu} \phi_{\e}) \partial_\nu u, \partial_\nu u \rangle_{L^2(\partial M)}.
\end{aligned}
\end{equation*}

Choose $\alpha = 4C_3$, then the above equation becomes
\begin{equation*}
\begin{aligned}
\lVert \mathcal{L}_{\phi_{\e}} u \rVert^2_{L^2(M)} \geq  &\frac{h^2}{\e}\left( 1-\e C_2 - \frac{h^2}{\alpha^2 \e} \right)\lVert u \rVert^2_{L^2(M)}
+ \frac{h^2}{\e}\left( \frac{2}{\alpha} - \e C_2 \right)\lVert h\nabla u\rVert^2_{L^2(M)} \\
&- 2h^3 \langle (\partial_{\nu} \phi_{\e}) \partial_\nu u, \partial_\nu u \rangle_{L^2(\partial M)}.
\end{aligned}
\end{equation*}
Choosing $\e = min \{ \frac{1}{4C_2}, \frac{1}{\alpha C_2} \}$ and $h_0$ so that it satisfies all the earlier restrictions as well as $\frac{h_0^2}{\alpha^2 \e} \leq \frac{1}{4}$. Hence we have the boundary Carleman estimate for $\mathcal{L}_{\phi_{\e}}$ on $M$ as:
\begin{equation*}
C\lVert \mathcal{L}_{\phi_{\e}} u \rVert^2_{L^2(M)}
\geq  \frac{h^2}{\e}\left(\lVert u \rVert^2_{L^2(M)} + \lVert h\nabla u\rVert^2_{L^2(M)}\right)
- 2h^3 \langle (\partial_{\nu} \phi_{\e}) \partial_\nu u, \partial_\nu u \rangle_{L^2(\partial M)}.
\end{equation*}
\end{proof}

We now prove a proposition which will help us to modify the above estimate to take care of the lower order perturbations.

\begin{proposition}\label{bdy_Carleman_estimate_proposition}
	Let $(M,g)$, $\mathcal{L}_{B,q}$ be as before and $\phi = \pm x_1$. There is a constant $C>0$ such that whenever $0<h$ is small and $u \in C^{\infty}(M)$ with $u|_{\partial M} \equiv 0$, one has
	\begin{equation}\label{bdy_Carleman_estimate}
	\begin{aligned}
	h^2&\lVert e^{\phi/h}\mathcal{L}_{B,q} u \rVert^2_{L^2(M)} + 2h \left\langle \lvert\partial_{\nu} \phi\rvert \partial_\nu (e^{\phi/h}u), \partial_\nu (e^{\phi/h}u) \right\rangle_{\{\partial_{\nu}\phi \geq 0 \}}\\
	&\geq  \lVert u \rVert^2_{L^2(M)} + \lVert h\nabla u\rVert^2_{L^2(M)}
	+ 2h \left\langle \lvert\partial_{\nu} \phi\rvert \partial_\nu (e^{\phi/h}u), \partial_\nu (e^{\phi/h}u) \right\rangle_{\{\partial_{\nu}\phi \leq 0 \}}.
	\end{aligned}
	\end{equation}
\end{proposition}

\begin{proof}
	We observe that (from \eqref{step 1})
	\begin{equation*}
	\mathcal{L}_{\phi_{\e}B,q} = \mathcal{L}_{\phi_{\e}} - 2ih^2 e^{\phi_{\e}/h}\langle B, d (e^{-\phi_{\e}/h}\cdot) \rangle_g + h^2\tilde{q}.
	\end{equation*}
	
	Hence,
	\begin{equation*}
	\begin{aligned}
	\left\lVert \mathcal{L}_{\phi_{\e},B,q} u \right\rVert^2_{L^2(M)}
	\geq & \lVert \mathcal{L}_{\phi_{\e}} u \rVert^2_{L^2(M)}\\
	&- \left(\lVert he^{\phi_{\e}/h}\langle B, hd(e^{-\phi_{\e}/h}u) \rangle_g \rVert^2_{L^2(M)} + \lVert h^2\tilde{q} u \rVert^2_{L^2(M)}\right)
	\end{aligned}
	\end{equation*}
	
	Now observe that the term
	\begin{equation*}
	\begin{aligned}
	&\lVert he^{\phi_{\e}/h}\langle B, hd(e^{-\phi_{\e}/h}u) \rangle \rVert_{L^2(M)}\\
	\leq& h\lVert \langle B, hdu \rangle_g \rVert_{L^2(M)} + h^2\lVert \langle B, d\phi_{\e} \rangle_g u \rVert_{L^2(M)}\\
	\leq& Ch \lVert B \rVert_{W^{1,\infty}(M)} \lVert u \rVert_{H^{1}(M)}
	+ Ch \lVert B \rVert_{L^{\infty}(M)} \lVert u \rVert_{L^2(M)}\\
	\leq& Ch \lVert u \rVert_{H^{1}(M)}.
	\end{aligned}
	\end{equation*}
	So, we get
	\begin{equation*}
	\begin{aligned}
	\left\lVert \mathcal{L}_{\phi_{\e},B,q} u \right\rVert^2_{L^2(M)}
	\geq& C\frac{h^2}{\e}\lVert u \rVert^2_{H^1(M)}
	- 2h^3 \langle (\partial_{\nu} \phi_{\e}) \partial_\nu u, \partial_\nu u \rangle_{\partial M}\\
	&- Ch^2 \lVert u \rVert^2_{H^1(M)} - Ch^4\lVert \tilde{q} \rVert^2_{L^{\infty}(M)} \lVert u \rVert^2_{L^2(M)}\\
	\geq& C\frac{h^2}{\e}\lVert u \rVert^2_{H^1(M)}
	- 2h^3 \langle (\partial_{\nu} \phi_{\e}) \partial_\nu u, \partial_\nu u \rangle_{\partial M},
	\end{aligned}
	\end{equation*}
	for $\e>0$ and $h>0$ small enough.
	
	Which implies
	\begin{equation*}
	\begin{aligned}
	\left\lVert \mathcal{L}_{\phi_{\e},B,q} u \right\rVert^2_{L^2(M)} &+ 2h^3 \langle (\partial_{\nu} \phi_{\e}) \partial_\nu u, \partial_\nu u \rangle_{ \{ \partial_{\nu}\phi_{\e} \geq 0 \} }\\
	\geq& C\frac{h^2}{\e}\lVert u \rVert^2_{H^1}
	- 2Ch^3 \langle (\partial_{\nu} \phi_{\e}) \partial_\nu u, \partial_\nu u \rangle_{ \{ \partial_{\nu}\phi_{\e} \leq 0 \} }.
	\end{aligned}
	\end{equation*}
	
	Here we observe that
\begin{equation*}
\begin{aligned}
	\{x \in \partial M : \partial_{\nu}\phi_{\e}(x) \geq 0 \} = \{ x \in \partial M: \partial_{\nu}\phi(x) \geq 0 \},\\
	 \mbox{and}\quad  \{x \in \partial M : \partial_{\nu}\phi_{\e}(x) \geq 0 \} = \{x \in \partial M : \partial_{\nu}\phi(x) \geq 0 \}.
\end{aligned}
\end{equation*}
	
	Moreover $e^{\phi_{\e}/h} = e^{\phi/h} e^{\phi^2/2\e}$ and for fixed $\e>0$ there exists $C>0$ such that
	\begin{equation*}
	\frac{1}{C} \leq \lVert e^{\phi^2/2\e}\rVert_{W^{2,\infty}(M)} \leq C.
	\end{equation*}
	Hence,
	\begin{equation*}
	\begin{aligned}
	\left\lVert h^2e^{\phi/h}\mathcal{L}_{B,q} e^{-\phi/h}u \right\rVert^2_{L^2(M)} &+ 2h^3 \langle (\partial_{\nu} \phi) \partial_\nu u, \partial_\nu u \rangle_{ \{ \partial_{\nu}\phi \geq 0 \} }\\
	\geq& Ch^2\lVert u \rVert^2_{H^1}
	- 2Ch^3 \langle (\partial_{\nu} \phi) \partial_\nu u, \partial_\nu u \rangle_{ \{ \partial_{\nu}\phi \leq 0 \} }.
	\end{aligned}
	\end{equation*}
	
	Now replacing $e^{-\phi/h}u$ by $v$ we see
	\begin{equation*}
	\begin{aligned}
	\left\lVert h^2e^{\phi/h}\mathcal{L}_{B,q} v \right\rVert^2_{L^2(M)} &+ 2h^3 \langle (\partial_{\nu} \phi) \partial_\nu (e^{\phi/h}v), \partial_\nu (e^{\phi/h}v) \rangle_{ \{ \partial_{\nu}\phi \geq 0 \} }\\
	\geq& Ch^2\lVert e^{\phi/h}v \rVert^2_{H^1}
	- 2Ch^3 \langle (\partial_{\nu} \phi) \partial_\nu (e^{\phi/h}v), \partial_\nu (e^{\phi/h}v) \rangle_{ \{ \partial_{\nu}\phi \leq 0 \} }.
	\end{aligned}
	\end{equation*}
\end{proof}

\section{Interior Carleman estimate}\label{int_Car_est}
In this section we prove a Carleman estimate for $u\in C^{\infty}(M)$ with $u|_{\partial M}$ supported in an open subset of $\partial M$. On bounded Euclidean domains, the proof is done in \cite{chung_2014}. We follow very closely the ideas of \cite{chung_2014} in the derivation of the Carleman estimate in this section.

We consider small open sets in $M$ where the Riemannian metric $g$ is nearly Euclidean after a suitable change of coordinates. We first prove the estimate on those open sets and later we patch it up over $M$ using a partition of unity.
In each coordinate patch, we use techniques similar to the case of the Euclidean domain, as in \cite{chung_2014}.
Due to the fact that our metric is non Euclidean, we encounter error terms, and in \eqref{step_3}, \eqref{step_9} we estimate the error terms.
Here we crucially use the fact that the metric $g$ is close to Euclidean on these coordinate patches.

Let us now fix $\phi(x) = x_1$ and recall that the operator $\mathcal{L}_{\phi_{\e}}$ is defined as
\begin{equation*}
\mathcal{L}_{\phi_{\e}} = -e^{\phi_{\e}/h}h^2\Delta_ge^{-\phi_{\e}/h}
\end{equation*}
and recall that $\mathcal{L}_{\phi,B,q} = e^{\phi/h}h^2\mathcal{L}_{B,q}e^{-\phi/h}$ on $\mathbb{R}_{+} \times M_0$.
Let $u \in C^{\infty}_c(M)$, then for small $\e>0$ and small enough $h \in (0,\e)$ the Estimate \eqref{c_E_0} implies
\begin{equation}\label{initial carleman est on omega tilde}
C\lVert \mathcal{L}_{\phi_{\e}} u \rVert_{L^2(M)}
	\geq 	\frac{h}{\sqrt{\e}} \lVert u \rVert_{H^1(M)}.
\end{equation}

Without loss of generality one can assume that $M \subset \mathbb{R}_{+} \times M_0$.
As $\Gamma_D \supset \partial_{+}M$ so there is $\delta>0$ so that
\begin{equation*}
\{x\in \partial M : \partial_{\nu}\phi > -2\delta \} \subset \Gamma_{D}.
\end{equation*}
Consider $E = \{x\in \partial M : \partial_{\nu}\phi \leq -\delta \}$ so that for any function $u$ vanishing on $E$ will imply $u|_{\partial M}$ is supported in $\Gamma_{D}$.
Let us consider a compact domain $\Omega \subset (\mathbb{R}_{+}\times(M_0,g_0))$ so that $M \subset \Omega$ and $E \subset \partial \Omega$.
Our aim is to prove the following estimate holds for all $u \in C^{\infty}_{c}(\Omega)$ and for $0<h<\e$ sufficiently small:
\begin{equation}\label{Int_Car_Est}
\frac{h}{\sqrt{\e}}\lVert u \rVert_{L^2(M)} \leq C\lVert \mathcal{L}_{\phi_{\e}}u \rVert_{H^{-1}(\Omega)}
\end{equation}
and subsequently the following estimate
\begin{equation*}
h\lVert u \rVert_{L^2(M)} \leq C\lVert \mathcal{L}_{\phi,B,q}u \rVert_{H^{-1}(\Omega)}.
\end{equation*}
From now on by $\delta^{ij}$, we denote the Kronecker delta and $\delta$ without superscripts is a small parameter.
We will start with proving the Carleman estimate for the following special case.

\subsection{Estimate for the special case}
Here we assume that
\begin{enumerate}
\item $M_0 \subset \mathbb{R}^{n-1}$ and the metric $g_0$ associated with $M_0$ is such that
\begin{equation}\label{step_10}
\lvert g_0^{jk} - \delta^{jk} \rvert < \delta,
\end{equation}
for some $\delta>0$ small.

\item The set $E \subset \partial M \cap \partial \Omega$ can be can be thought of as contained in the graph of a smooth function $f: \mathbb{R}^{n-1} \to \mathbb{R}_{+}$. That is
\begin{equation*}
E \subset \{ (f(\p{x}),\p{x}) : \p{x} \in M_0 \subset \mathbb{R}^{n-1} \}.
\end{equation*}
\item The function $f$ is so that
\begin{equation*}
\lvert \nabla_{g_0}f - Ke_{n-1} \rvert \leq \mu,
\end{equation*}
where $\lvert\cdot\rvert$ denotes the Euclidean distance, $K$ is some positive real number, $\mu>0$ small and $e_{n-1} \in \mathbb{R}^{n-1}$ defined as $e_{n-1} = (0,\dots,0,1)$.
\end{enumerate}

Here we would like to mention that the first assumption is motivated by the techniques in the proof of the Carleman estimate in \cite{Chung2013}.
Without loss of generality one can make the assumption that the domain $\Omega$ belongs to the space
\begin{equation*}
A_0 = \{ (x_1,\p{x})\in \mathbb{R}\times \mathbb{R}^{n-1} : x_1 \geq f(\p{x}) \}.
\end{equation*}

Let us make a change of variable $\sigma: (x_1,\p{x}) \mapsto (x_1 - f(\p{x}),\p{x})$. Under this change of coordinates we get $\tilde{E} = \sigma(E)$ is a subset of the plane $x_1 = 0$ and the domain $\tilde{\Omega} = \sigma(\Omega)$, sits in $\tilde{A}_0 = \{ (x_1,\p{x}): x_1 \geq 0 \}$. Let us denote $\tilde{w}(x_,\p{x})$ in place of $w(\sigma^{-1}(x_1,\p{x})) = w(x_1+f(\p{x}),\p{x})$ for any function $w$ defined on $\Omega$ from now on.

A calculation shows that the form of the operator $\mathcal{L}_{\phi_{\e}}$ in this new coordinate system is
\begin{equation*}
\begin{aligned}
-\widetilde{\mathcal{L}_{\phi_{\e}}u}({x}_1,\p{{x}})
=& \left[ 1 + \lvert \nabla_{\p{x}}f(\p{x})\rvert_{g_0}^2 \right]h^2 \partial_{x_1}^2 \tilde{u} \\
&- \left[ 2\left( 1+ \frac{hx_1}{\e} \right) + 2\langle \nabla_{\p{x}}f(\p{x}), h\nabla_{\p{x}} \rangle_{g_0} \right]h\partial_{x_1}\tilde{u} \\
&+ \left[\left( 1 + h\frac{x_1}{\e} \right)^2 + h^2 \Delta_{g_0} \right]\tilde{u} + hE_1\tilde{u},
\end{aligned}
\end{equation*}
where $E_1$ is a first order semiclassical differential operator with bounded coefficients.
We define the operator $\tilde{\mathcal{L}}_{\phi_{\e}}$ on $M$ as
\begin{equation*}
\begin{aligned}
\tilde{\mathcal{L}}_{\phi_{\e}}
=& \left[ 1 + \lvert \nabla_{\p{x}}f(\p{x})\rvert_{g_0}^2 \right]h^2 \partial_{x_1}^2  \\
&- \left[ 2\left( 1+ \frac{hx_1}{\e} \right) + 2\langle \nabla_{\p{x}}f(\p{x}), h\nabla_{\p{x}} \rangle_{g_0} \right]h\partial_{x_1} \\
&+ \left[\left( 1 + h\frac{x_1}{\e} \right)^2 + h^2 \mathcal{L} \right],
\end{aligned}
\end{equation*}
where $\mathcal{L} = g_0^{jk}\partial_{x_j x_k}$, here $2 \leq j,k\leq n$.
Then we have $ -\widetilde{\mathcal{L}_{\phi_{\e}}u} = \tilde{\mathcal{L}}_{\phi_{\e}} \tilde{u} + hE_1\tilde{u}$.

Observe that due to our assumption on $g$, we can extend $g$ smoothly onto $\mathbb{R}^{n-1}$ having the property $g^{jk} = \delta^{jk}$ outside an open neighborhood of $\tilde{\Omega}$.
Let $\alpha, \gamma_f$ be a smooth functions on $\mathbb{R}^{n}$ so that $\alpha$ agrees with $\left( 1 + h\frac{x_1}{\e} \right)$ on $\tilde{\Omega}$ and $\gamma_f = \lvert\nabla_{\p{x}} f(\p{x}) \rvert_{g_0}$ on $\tilde{\Omega}$. Let $\beta_f$ be a smooth vector field on $\mathbb{R}^n$ so that it agrees with $\nabla_{g_0} f$ on $\tilde{\Omega}$.

Having this setup, in this subsection, our aim is to prove the following proposition.

\begin{proposition}\label{Target_Proposition}
Let $\tilde{M}$, $\tilde{\Omega}$, $\tilde{\mathcal{L}}_{\phi_{\e}}$ be as above, then for small $\e>0$ and small enough $h \in (0,\e)$ one has
\begin{equation*}
C\frac{h}{\sqrt{\e}}\lVert w \rVert_{L^2(\R)}
\leq \lVert \tilde{\mathcal{L}}_{\phi_{\e}} w \rVert_{H^{-1}(\R)},\quad \forall w\in C^{\infty}_c(\tilde{\Omega}).
\end{equation*}
\end{proposition}

\begin{corollary}
Assuming the above notations, for $\e>0$ small and $h \in (0,\e)$ small enough we get
\begin{equation}\label{Target_estimate}
C\frac{h}{\sqrt{\e}}\lVert w \rVert_{L^2(\Omega)}
\leq \lVert \tilde{\mathcal{L}}_{\phi_{\e}} w \rVert_{H^{-1}(A_0)},\quad \forall w\in C^{\infty}_c(\Omega).
\end{equation}
\end{corollary}
\begin{proof}[Proof of the corollary]
First we state the following lemma borrowed from \cite{chung_2014}.
\begin{lemma}\label{transformation norm equiv lema}
For any function $w \in C^\infty_c(\Omega)$ we have following two relations
\begin{equation*}
\lVert w \rVert_{L^2(\Omega)} \simeq \lVert \tilde{w} \rVert_{L^2(\tilde{\Omega})}, \qquad \quad
\lVert w \rVert_{H^1(\Omega)} \simeq \lVert \tilde{w} \rVert_{H^1(\tilde{\Omega})}.
\end{equation*}
\end{lemma}
Here $a \simeq b$ means for some constants $c_1,c_2>0$ one has $c_1 a \leq b \leq c_2 a$.
Note that using duality one can prove that
$\lVert w \rVert_{H^{-1} (A_0)} \simeq \lVert \tilde{w} \rVert_{H^{-1} (\R)}$.

Having the Proposition \ref{Target_Proposition} we get for small enough $0<h<\e$,
\begin{equation*}
C\frac{h}{\sqrt{\e}}\lVert \tilde{w} \rVert_{L^2(\tilde{\Omega})}
\leq \lVert \tilde{\mathcal{L}}_{\phi_{\e}} \tilde{w} \rVert_{H^{-1}(\R)},\quad \forall \tilde{w} \in C^{\infty}_c(\tilde{\Omega}).
\end{equation*}

Let $w \in C^{\infty}_c(\Omega)$, then using the relation $\tilde{\mathcal{L}}_{\phi_{\e}} \tilde{u} + hE_1\tilde{u} = -\widetilde{\mathcal{L}_{\phi_{\e}}u}$, we get
\begin{equation*}
\begin{gathered}
C\frac{h}{\sqrt{\e}}\lVert \tilde{w} \rVert_{L^2(\tilde{\Omega})}
\leq C\lVert \tilde{\mathcal{L}}_{\phi_{\e}} \tilde{w} \rVert_{H^{-1}(\R)}
\leq \lVert \widetilde{\mathcal{L}_{\phi_{\e}} w} \rVert_{H^{-1}(\R)}
	+ h\lVert E_1 \tilde{w} \rVert_{H^{-1}(\R)}\\
\implies C\frac{h}{\sqrt{\e}}\lVert \tilde{w} \rVert_{L^2(\tilde{\Omega})}
\leq \lVert \widetilde{\mathcal{L}_{\phi_{\e}} w} \rVert_{H^{-1}(\R)}
	+ h\lVert \tilde{w} \rVert_{L^2(\R)}\\
\implies C\frac{h}{\sqrt{\e}}\lVert \tilde{w} \rVert_{L^2(\tilde{\Omega})}
\leq \lVert \widetilde{\mathcal{L}_{\phi_{\e}} w} \rVert_{H^{-1}(\R)}
\leq \lVert \mathcal{L}_{\phi_{\e}} w \rVert_{H^{-1}(A_0)}, \quad \mbox{for }h>0 \mbox{ small}.
\end{gathered}
\end{equation*}

On the other hand we get
\begin{equation*}
\frac{h}{\sqrt{\e}}\lVert w \rVert_{L^2(\Omega)}
\leq C\frac{h}{\sqrt{\e}}\lVert \tilde{w} \rVert_{L^2(\tilde{\Omega})}
\end{equation*}

Hence combining the above two estimates, for $0<h<\e$ small enough, we get
\begin{equation*}
\frac{h}{\sqrt{\e}}\lVert w \rVert_{L^2(\Omega)}
\leq \lVert \mathcal{L}_{\phi_{\e}} w \rVert_{H^{-1}(A_0)}, \quad \forall w \in C^{\infty}_c(\Omega).
\end{equation*}
\end{proof}

We will prove the Proposition \ref{Target_Proposition} by dividing it into two cases for small and large frequencies. To define our notion of small and large frequencies let us write $\mathcal{S}(\R)$ to be the restrictions of Schwartz functions to $\R$.
Let $K>0$ be as in the assumption-3.
We define $r_1,r_2,\delta_1,\delta_2$ be such that
\begin{equation*}
\frac{K^2}{1 + K^2} < r_1 < r_2 \leq \frac{1}{2} + \frac{K^2}{2(1+K^2)} < 1,
\end{equation*}
and $\delta_1,\delta_2>0$. Define a smooth cutoff function $\rho \in C^{\infty}_c(\mathbb{R}^{n-1})$ so that
\begin{equation*}
\begin{aligned}
\rho(\xi) =
			\begin{cases} 	0, 		& \mbox{if } \lvert \xi \rvert^2 > r_2 \mbox{ or } \lvert \xi_{n-1} \rvert>\delta_2 \\
							1, 		& \mbox{if } \lvert \xi \rvert^2 \leq r_1 \mbox{ and } \lvert \xi_{n-1} \rvert\leq \delta_1.
			\end{cases}
\end{aligned}
\end{equation*}

Let us denote $\hat{v}(x_1,\xi)$ as semiclassical Fourier transform of a function $v$ in the $\p{x}$ variables. For any $w\in C^{\infty}_c(\tilde{\Omega})$ define $\hat{w}_s = \rho(\xi)\hat{w}$ and $\hat{w}_l = (1-\rho(\xi))\hat{w}$.

Here we state the two lemmas for small and large frequencies.

\begin{lemma}[Small frequency lemma]\label{small freq lemma}
There exists $r_1<r_2$ and $\delta_1<\delta_2$ such that
for $0<h<\e$ small enough and for all $w \in C^{\infty}_c(\tilde{\Omega})$ we have
\begin{equation}\label{small freq estimate}
C\frac{h}{\sqrt{\e}}\lVert w_s \rVert_{L^2(\R)}
\leq \lVert \tilde{\mathcal{L}}_{\phi_{\e}} w_s \rVert_{H^{-1}(\R)}
+ h\lVert w \rVert_{L^2(\tilde{M})}.
\end{equation}
\end{lemma}

For the other part $w_l$ we have the following lemma
\begin{lemma}[Large frequency lemma]\label{large freq lemma}
For $0<h<\e$ small enough and for all $w \in C^{\infty}_c(\tilde{\Omega})$ we have
\begin{equation}\label{big freq estimate}
C\frac{h}{\sqrt{\e}}\lVert w_l \rVert_{L^2(\R)}
\leq \lVert \tilde{\mathcal{L}}_{\phi_{\e}} w_l \rVert_{H^{-1}(\R)}
+ h\lVert w \rVert_{L^2(\tilde{M})}
\end{equation}
\end{lemma}

Here we refer to the calculation in \cite[Section 4]{chung_2014} to obtain the estimate in Proposition \ref{Target_Proposition}.
In the next two subsections we will prove the above mentioned lemmas.

\subsection{Small frequency case (Lemma \ref{small freq lemma}):}
Following the approach of \cite{chung_2014} (Section 3) let us consider the function
\begin{equation*}
F:\mathbb{R}^{n-1}\rightarrow\mathbb{C}, \qquad \text{defined as}
\end{equation*}
\begin{equation*}
\overline{F(\xi)} = \frac{1}{1+\lvert K \rvert^2}\left( 1+iK\xi_{n-1} + \sqrt{2iK\xi_{n-1} - (K\xi_{n-1})^2 + (1+\lvert K\rvert^2)\lvert \xi \rvert^2 - \lvert K \rvert^2 } \right),
\end{equation*}
where the branch of square root is considered with the non negative imaginary axis.

Observe that $F$ is smooth except where
\begin{equation*}
\tilde{F}(\xi) = 2iK\xi_{n-1} - (K\xi_{n-1})^2 + (1 + \lvert K \rvert^2)\lvert \xi \rvert^2 - \lvert K \rvert^2
\end{equation*}
belongs to the non negative real axis. That is, this happens when
\begin{equation*}
\xi_{n-1} = 0 \qquad and \quad \lvert \xi \rvert^2 \geq \frac{\lvert K \rvert^2}{1 + \lvert K \rvert^2}.
\end{equation*}
The discontinuity in this case is a jump of size $2\sqrt{(1+\lvert K \rvert^2)\lvert \xi\rvert^2 - \lvert K \rvert^2}$.
But, on the support of $\rho$, we have $\lvert \xi \rvert \leq r_2$.
So, for $r_2$ close to $\frac{\lvert K \rvert^2}{1 + \lvert K \rvert^2}$ the jump is small.
Therefore, on the support of $\rho$, we can approximate $F$ by a smooth function $F_s$ such that, for small $\delta$
\begin{equation*}
\lvert F(\xi) - F_{s}(\xi) \rvert \leq \delta.
\end{equation*}

Observe that the bound on derivatives of $F_s$ may depend on $\delta$. Note that without loss of generality, here we can use the same $\delta>0$ as in Equation \eqref{step_10}.

Now, let us calculate bounds on $F_s$.
On the support of $\rho$,
\begin{equation*}
\Im(\tilde{F}) \in [-2K\delta_2 , 2k \delta_2].
\end{equation*}

Choose $r_2$ so close to $\frac{\lvert K \rvert^2}{1 + \lvert K \rvert^2}$ so that $(1+\lvert K\rvert^2)r_2 - \lvert K \rvert^2 \leq \delta_2$. Then
\begin{equation*}
\begin{aligned}
\Re(\tilde{F}) &= -(K\xi_{n-1})^2 - \lvert K \rvert^2 + (1 + \lvert K \rvert^2)\lvert \xi \rvert^2\\
&\leq -(K\xi_{n-1})^2 + \delta_2 \leq \delta_2.
\end{aligned}
\end{equation*}

Hence, as the bounds of both of $\Re(\tilde{F})$ and $\Im (\tilde{F})$ depends on $\delta_2$, we can choose $\delta_2$ small so that on the support of $\rho(\xi)$, $\lvert\Re(\sqrt{\tilde{F}})\rvert \leq 1/3$.
Hence, we can have both of $\Re(F_s)$ and $\lvert F_s\rvert$ are $> \frac{1}{2+2\lvert K\rvert^2}$, for $\delta$ small.

One can extend $F_s$ on whole $\mathbb{R}^{n-1}$ such that it satisfies $\Re(F_s), \lvert F_s\rvert > \frac{1}{2 + 2\lvert K \rvert^2}$, for all $\xi$ and $\Re(F_s), \lvert F_s\rvert \simeq (1+\lvert \xi \rvert)$ for $\lvert \xi \rvert$ large.

For $u\in \mathcal{S}(\R)$ define the operator $J_s$ by
\begin{equation*}
\widehat{J_s u}(x_1,\xi) = \left( {F_s(\xi)} + h\partial_{x_1} \right)\hat{u}(x_1,\xi).
\end{equation*}
The adjoint of the above operator is defined as
\begin{equation*}
\widehat{J^*_s u}(x_1,\xi) = \left( {\overline{F_s(\xi)}} - h\partial_{x_1} \right)\hat{u}(x_1,\xi).
\end{equation*}
These operators have right inverses defined as
\begin{equation*}
\begin{aligned}
\widehat{J^{-1}_s}u(x_1,\xi) =& h^{-1} \int_{0}^{x_1} \hat{u}(t,\xi)e^{\frac{t-x_1}{h}F_s(\xi)} \ dt,\\
\mbox{and}\qquad
\widehat{{J^*_s}^{-1}}u(x_1,\xi) =& h^{-1} \int_{x_1}^{\infty} \hat{u}(t,\xi)e^{\frac{x_1-t}{h}\overline{F_s(\xi)}} \ dt.
\end{aligned}
\end{equation*}
Each of the above is well defined functions in $\mathcal{S}(\R)$.
Note that here we can use the same form of $J_s$, as for the Euclidean case in \cite{chung_2}, mainly because we assume that the Riemannian metric is approximately Euclidean.
For boundedness properties of the above operators we state the following lemma from \cite{chung_2014}.

\begin{lemma}\label{boundedness lemma}
$J_s, J^*_s, J_s^{-1}, {J_s^*}^{-1}$ extends as bounded maps
\begin{equation*}
\begin{aligned}
J_s, J_s^* : H^1(\R) \rightarrow L^2(\R),
\quad\mbox{and}\quad
J_s^{-1}, {J_s^*}^{-1} : L^2(\R) \rightarrow H^1(\R).
\end{aligned}
\end{equation*}
Moreover, the extensions of $J^*_s$ and ${J^*_s}^{-1}$ are isomorphisms.
\end{lemma}

Now we state the following commutativity lemma for the operator $J_s$. We will not present the proof here, as it follows from similar arguments as in the proof of \cite[Lemma 5.2]{chung_2014}.
\begin{lemma}\label{commutativity with 2nd order op}
Let $w \in \mathcal{S}(\R)$, if $Q$ is a second order semiclassical differential operator with bounded and smooth coefficients, then
\begin{equation*}
\lVert (J_sQ - QJ_s)w \rVert_{H^{-1}(\R)} \leq hc\lVert w \rVert_{H^1(\R)}.
\end{equation*}
Moreover, if $\chi \in \mathcal{S}(\R)$, then
\begin{equation*}
\Vert J_s\chi J_s^{-1} w \rVert_{L^2(\R)} \geq \lVert \chi w \rVert_{L^2(\R)} - hc\lVert w\rVert_{L^2(\R)}.
\end{equation*}
\end{lemma}

We define a function $\g$ as
\begin{equation*}
\hat{\g}(x_1,\xi) = \frac{2\Re F_s(\xi)}{h}\int_{0}^{\infty} \hat{u}(t,\xi)e^{-\frac{1}{h}\left( t\overline{F_s(\xi)} + x_1F_s(\xi) \right)} dt.
\end{equation*}
Here $\Re(z) \equiv \Re z$ means the real part of the complex number $z$.
Observe that $J_s \g = 0$. For this $\g$ we have the following lemma.

\begin{lemma}\label{lemma for g}
For $u\in \mathcal{S}(\R)$ and $\g$ as above, we have
\begin{equation*}
\lVert J_s u \rVert_{H^{-1}(\R)} \simeq \lVert u - \g\rVert_{L^2(\R)}.
\end{equation*}
\end{lemma}
\begin{proof}
Let $u\in \mathcal{S}(\R)$ and $\g$ is defined as above. First let us show that indeed $\g \in L^2(\R)$.
\begin{equation*}
\begin{aligned}
\int_{0}^{\infty}& \lvert \hat{\g}(x_1,\xi)\rvert^2 dx_1 \\
=& \left|\frac{2\Re F_s(\xi)}{h}\right|^2 \int_{0}^{\infty} \left| \int_{0}^{\infty} \hat{u}(t,\xi)e^{-\frac{1}{h}\left( t\overline{F_s(\xi)} + x_1F_s(\xi) \right)} dt\right|^2 dx_1\\
\leq& C\left|\frac{2\Re F_s(\xi)}{h}\right|^2 \int_{0}^{\infty} \left| \hat{u}(t,\xi)\right|^2 dt \int_{0}^{\infty}\left|e^{-\frac{1}{h}\left( t\overline{F_s(\xi)} + x_1F_s(\xi) \right)}\right|^2 dt\ dx_1\\
\leq& C\left|\frac{2\Re F_s(\xi)}{h}\right|^2 \int_{0}^{\infty} \left| \hat{u}(t,\xi)\right|^2 dt \int_{0}^{\infty}\left|e^{-\frac{1}{h}\left( t\overline{F_s(\xi)}\right)}\right|^2 dt
\int_{0}^{\infty}\left|e^{-\frac{1}{h}\left( x_1F_s(\xi)\right)}\right|^2 dx_1\\
\leq& C\left|\frac{2\Re F_s(\xi)}{h}\right|^2 \int_{0}^{\infty} \left| \hat{u}(t,\xi)\right|^2 dt \int_{0}^{\infty}e^{-\frac{2}{h}\left( t\Re\overline{F_s(\xi)}\right)} dt
\int_{0}^{\infty}e^{-\frac{2}{h}\left( x_1\Re F_s(\xi)\right)} dx_1\\
\leq& C\int_{0}^{\infty} \left| \hat{u}(t,\xi)\right|^2 dt\\
\end{aligned}
\end{equation*}
So, $\g \in L^2(\R)$ and $\lVert \g\rVert^2_{L^2(\R)} \leq C \lVert u\rVert^2_{L^2(\R)}$.

Now, observe that $\widehat{J_s\g}(x_1,\xi) = 0$.
Hence,
\begin{equation*}
\begin{aligned}
\lVert J_s u \rVert_{H^{-1}(\R)}
&= \sup_{w\in H^1_0(\R) \setminus \{0\} } \frac{\lvert \langle J_s u, w \rangle\rvert}{\lVert w\rVert_{H^1_{0}(\R)}}
= \sup_{w\in H^1_0(\R) \setminus \{0\} } \frac{\lvert \langle J_s (u-\g), w \rangle\rvert}{\lVert w\rVert_{H^1_{0}(\R)}},
\\
\implies \lVert J_s u \rVert_{H^{-1}(\R)}
&\simeq \sup_{w\in H^1_0(\R) \setminus \{0\} } \frac{\lvert \langle (u-\g), J_s^*w \rangle\rvert}{\lVert J_s^*w\rVert_{L^2(\R)}}
\leq C\lVert u-\g \rVert_{L^2(\R)}.
\end{aligned}
\end{equation*}

For the other part observe that our choice of $\g$ implies ${J^*_s}^{-1}(u-\g) \in H^1_0(\R)$. If $u=\g$ then the proof is complete, else we take $w={J^*_s}^{-1}(u-\g)$ and obtain
\begin{equation*}
\lVert J_s u \rVert_{H^{-1}(\R)} \geq C\lVert u-\g \rVert_{L^2(\R)}.
\end{equation*}
\end{proof}\vspace*{-8pt}

Let $\chi(x_1,\p{x}) \in C^{\infty}(\R)$ be a cutoff function with $\chi = 1$ on $\tilde{M}$ and $\chi = 0$ outside $\tilde{\Omega}$. When $w \in C_c^{\infty}(\tilde{M})$ then $w_s \in \mathcal{S}(\R)$, and supported away from $x_1 = 0$. Hence $J_s^{-1}w_s \in \mathcal{S}(\R)$ is also supported away from $x_1 = 0$ and so $\chi J_s^{-1}w_s \in C_c^{\infty}(\tilde{\Omega})$. Hence, by \eqref{initial carleman est on omega tilde} we have
\begin{equation*}
\begin{aligned}
\hspace*{-4pt}&\frac{h}{\sqrt{\e}} \lVert \chi J_s^{-1}w_s \rVert_{H^1(\tilde{\Omega})}
\leq C \lVert \tilde{\mathcal{L}}_{\phi_{\e}}(\chi J_s^{-1}w_s) \rVert_{L^2(\tilde{\Omega})}\\
\hspace*{-4pt}\implies&\frac{h}{\sqrt{\e}} \lVert J_s \chi J_s^{-1}w_s \rVert_{L^2(\tilde{\Omega})}
\leq C \lVert \tilde{\mathcal{L}}_{\phi_{\e}}(\chi J_s^{-1}w_s) \rVert_{L^2(\R)}\quad\qquad \mbox{using Lemma \ref{boundedness lemma},}\\
\hspace*{-4pt}\implies&\frac{h}{\sqrt{\e}} \lVert \chi w_s \rVert_{L^2(\tilde{\Omega})}
\leq C \lVert \tilde{\mathcal{L}}_{\phi_{\e}}(\chi J_s^{-1}w_s) \rVert_{L^2(\R)} + c\frac{h^2}{\e}\lVert w_s\rVert_{L^2(\R)}, \mbox{ using Lemma \ref{commutativity with 2nd order op}}.
\end{aligned}
\end{equation*}
Now, observe that $\chi = 1$ on the support of $w$. Hence,
\begin{equation*}
\chi w_s = \chi P w = Pw + \mathcal{O}(h^{\infty})E_0w = w_s + \mathcal{O}(h^{\infty})E_0w
\end{equation*}
where $E_0$ is a order $0$ pseudo differential operator on $\mathbb{R}^{n-1}$.
Which implies
\begin{equation}\label{step_4}
\begin{aligned}
\frac{h}{\sqrt{\e}} \lVert w_s \rVert_{L^2(\R)}
\leq & C \lVert \tilde{\mathcal{L}}_{\phi_{\e}}(\chi J_s^{-1}w_s) \rVert_{L^2(\R)} + c\frac{h^2}{\e}\lVert w_s\rVert_{L^2(\R)} + \mathcal{O}(h^{\infty})\lVert w \rVert_{L^2(\R)}\\
\leq &C \lVert \chi \tilde{\mathcal{L}}_{\phi_{\e}}(J_s^{-1}w_s) \rVert_{L^2(\R)} + h\lVert J_s^{-1} w_s \rVert_{H^1(\R)} + \mathcal{O}(h^{\infty})\lVert w \rVert_{L^2(\R)}\\
\leq &C \lVert \tilde{\mathcal{L}}_{\phi_{\e}}(J_s^{-1}w_s) \rVert_{L^2(\R)} + \mathcal{O}(h^{\infty})\lVert w \rVert_{L^2(\R)}.
\end{aligned}
\end{equation}
Now we want to take $u = \tilde{\mathcal{L}}_{\phi_{\e}}(J_s^{-1}w_s)$ in Lemma \ref{lemma for g} so that we get the following bound
\begin{equation*}
\lVert \tilde{\mathcal{L}}_{\phi_{\e}}(J_s^{-1}w_s) \rVert_{L^2(\R)} \leq C \lVert J_s \tilde{\mathcal{L}}_{\phi_{\e}}J_s^{-1}w_s \rVert_{H^{-1}(\R)}.
\end{equation*}

For this now we will show that $\lVert \g \rVert_{L^2(\R)} \leq \frac{1}{2}\lVert u \rVert_{L^2(\R)}$.
For sake of notational simplicity, we will denote Fourier transform in $\p{x}$ variable of a function $v(x_1,\p{x})$ as $\mathcal{F}(v)(x_1,\xi)$ at some places in the calculation below. We will also use the symbol $\hat{v}(x_1,\xi)$ for the same purpose, where it seems convenient.

Writing $v = J^{-1}_s w_s$ we get
\begin{equation*}
\begin{aligned}
\hat{\g}(x_1,\xi)
=& \frac{2\Re F_s}{h}\int_{0}^{\infty} \widehat{\tilde{\mathcal{L}}_{\phi_{\e}}v}(t,\xi) e^{-\frac{1}{h}(t\overline{F_s(\xi)} + x_1F_s(\xi))} dt.
\end{aligned}
\end{equation*}

A calculation implies,
\begin{equation}\label{estimate g_hat}
\begin{aligned}
\hat{\g}
=& \frac{2\Re F_s}{h}\int_{0}^{\infty} \left[(1+K^2)h^2\partial^2_t - 2(1+iK\xi_{n-1})h\partial_t + 1-\lvert \xi \rvert^2\right]\\
& \hspace{8cm} \hat{v}(t,\xi)e^{-\frac{t\overline{F_s(\xi)} + x_1F_s(\xi)}{h}} dt\\
&-\frac{2\Re F_s}{h}\int_{0}^{\infty} 2\mathcal{F}\left( \langle (\beta_f - Ke_{n-1}) ,h\nabla_{\p{x}}h\partial_t v\rangle \right) (t,\xi) e^{-\frac{1}{h}(t\overline{F_s(\xi)} + x_1F_s(\xi))} dt\\
&+ \frac{2\Re F_s}{h}\int_{0}^{\infty} \mathcal{F}\left((\lvert \gamma_f\rvert^2 - K^2) h^2\partial_t^2 v\right) (t,\xi) e^{-\frac{1}{h}(t\overline{F_s(\xi)} + x_1F_s(\xi))} dt\\
&+\frac{2\Re F_s}{h}\int_{0}^{\infty} \left[2h\frac{t}{\e}-2h\frac{t}{\e} h\partial_t + h^2\frac{t^2}{\e^2} + h^2 \mathcal{F}(\mathcal{L} - \Delta_{\p{x}}) \right] \hat{v} (t,\xi) e^{-\frac{t\overline{F_s(\xi)} + x_1F_s(\xi)}{h}} dt.
\end{aligned}
\end{equation}

Now doing an integration by parts for the first term in the above identity, we get
\begin{equation*}
\begin{aligned}
\hat{\g}
=& \frac{2\Re F_s}{h}\int_{0}^{\infty} \left[(1+K^2)\overline{F_s(\xi)}^2 -2(1+iK\xi_{n-1})\overline{F_s(\xi)} + (1-\lvert \xi \rvert^2) \right]\\
& \hspace{8cm} \hat{v}(t,\xi) e^{-\frac{t\overline{F_s(\xi)} + x_1F_s(\xi)}{h}} dt\\
&-\frac{2\Re F_s}{h}\int_{0}^{\infty} 2\mathcal{F}\left( \langle (\beta_f - Ke_{n-1}) ,h\nabla_{\p{x}}h\partial_t v\rangle \right) (t,\xi) e^{-\frac{1}{h}(t\overline{F_s(\xi)} + x_1F_s(\xi))} dt\\
&+ \frac{2\Re F_s}{h}\int_{0}^{\infty} \mathcal{F}\left((\lvert \gamma_f\rvert^2 - K^2) h^2\partial_t^2 v\right) (t,\xi) e^{-\frac{1}{h}(t\overline{F_s(\xi)} + x_1F_s(\xi))} dt\\
&+\frac{2\Re F_s}{h}\int_{0}^{\infty} \left[2h\frac{t}{\e}-2h\frac{t}{\e} h\partial_t + h^2\frac{t^2}{\e^2} + h^2 \mathcal{F}(\mathcal{L} - \Delta_{\p{x}}) \right] \hat{v} (t,\xi) e^{-\frac{t\overline{F_s(\xi)} + x_1F_s(\xi)}{h}} dt.
\end{aligned}
\end{equation*}
Using the same estimation technique used in the proof of Lemma \eqref{lemma for g} we get
\begin{equation*}
\begin{aligned}
\lVert \hat{\g} \rVert^2_{L^2(\R)}
\leq& \left\lVert (1+K^2)\overline{F_s(\xi)}^2 -2(1+iK\xi_{n-1})\overline{F_s(\xi)} + (1-\lvert \xi \rvert^2)\hat{v}(x_1,\xi) \right\rVert^2_{L^2(\R)}\\
&+\left\lVert \langle (\beta_f - Ke_{n-1}) ,h\nabla_{\p{x}}h\partial_t v\rangle \right\rVert^2_{L^2(\R)}
+ \left\lVert (\lvert \gamma_f\rvert^2 - K^2) h^2\partial_t^2 v \right\rVert^2_{L^2(\R)}\\
&+ C_1\frac{h^2}{\e^2}\left\lVert h\partial_t v \right\rVert^2_{L^2(\R)}
 + C_2\frac{h^2}{\e^2}\left\lVert v \right\rVert^2_{L^2(\R)}
 + C_3\lVert h^2(\mathcal{L} - \Delta_{\p{x}})v \rVert^2_{L^2(\R)}.
\end{aligned}
\end{equation*}
Now, on the support of $\hat{w}_s(x_1,\xi)$ and hence on the support of $\hat{v}$ we have $\overline{F_s(\xi)}$ is an approximate solution of the equation
\begin{equation*}
(1+K^2)Z^2 -2(1+iK\xi_{n-1})Z + (1-\lvert \xi \rvert^2) = 0.
\end{equation*}
Hence,
\begin{equation*}
\left\lvert (1+K^2)\overline{F_s(\xi)}^2 -2(1+iK\xi_{n-1})\overline{F_s(\xi)} + (1-\lvert \xi \rvert^2) \right\rvert \leq C_\delta \lvert \overline{F_s(\xi)} \rvert,
\end{equation*}
where $C_{\mu}$ and $C_{\delta}$ are small when $\mu$ and $\delta$ are small.
Observe that
\begin{equation*}
C_3\left\lVert h^2(\mathcal{L} - \Delta_{\p{x}})v \right\rVert^2_{L^2(\R)}
\leq C_3\sum_{j,k=1}^{n}\left\lVert (g^{jk} - \delta^{jk})h^2\frac{\partial^2 v}{\partial x_j \partial x_k} \right \rVert^2_{L^2(\R)}
\leq C_{\delta} \lVert v \rVert^2_{H^2(\R)},
\end{equation*}
where $C_{\delta} \to 0$ as $\delta\to 0$. Here we use the assumption that the metric is close to Euclidean (see \eqref{step_10}).
Hence,
\begin{equation}\label{step_3}
\begin{aligned}
\lVert \hat{\g} \rVert^2_{L^2(\R)}
&\leq \left\lVert C_{\delta}\lvert F_s(\xi)\lvert\hat{v}(x_1,\xi) \right\rVert^2_{L^2(\R)} + (C_{\mu} + C_{\delta})\left\lVert v \right\rVert^2_{H^2(\R)}
+ C_1\frac{h^2}{\e^2}\left\lVert v \right\rVert^2_{H^1(\R)}\\
&\leq C^2_{\delta}\lVert v \rVert^2_{H^1(\R)} + (C_{\mu} + C_{\delta})\left\lVert v \right\rVert^2_{H^2(\R)}
+ C_1\frac{h^2}{\e^2}\left\lVert v \right\rVert^2_{H^1(\R)}.
\end{aligned}
\end{equation}
Notice that $C_{\mu}$, $C_{\delta}$ small whenever $\mu$ and $\delta$ is small respectively. Hence for small enough $h<\e$ we get
\begin{equation*}
\lVert \hat{\g} \rVert^2_{L^2(\R)} \leq (C_{\mu} + C_{\delta}) \left\lVert v \right\rVert^2_{H^2(\R)}.
\end{equation*}

To get an estimate of $\g$ in terms of $u$ recall that $u = \tilde{\mathcal{L}}_{\phi_{\e}} v$. Hence,
\begin{equation*}
\begin{aligned}
\lVert u \rVert^2_{L^2(\R)}
=& \lVert \tilde{\mathcal{L}}_{\phi_{\e}} v \rVert^2_{L^2(\R)}\\
\geq& \left\lVert \left[(1+K^2)h^2\partial^2_{x_1} -2(1+iK\xi_{n-1})h\partial_{x_1} + (1-\lvert \xi \rvert^2)\right] v(x_1,\xi) \right\rVert^2_{L^2(\R)}\\
&-\left\lVert \langle (\beta_f - Ke_{n-1}) ,h\nabla_{\p{x}}h\partial_t v\rangle \right\rVert^2_{L^2(\R)}
- \left\lVert (\lvert \gamma_f\rvert^2 - K^2) h^2\partial_t^2 v \right\rVert^2_{L^2(\R)}\\
&- C_1\frac{h^2}{\e^2}\left\lVert h\partial_t v \right\rVert^2_{L^2(\R)}
 - C_2\frac{h^2}{\e^2}\left\lVert v \right\rVert^2_{L^2(\R)} - C_{\delta}\lVert v \rVert^2_{H^2(\R)}.
\end{aligned}
\end{equation*}

Now, for $\lvert \xi \rvert < 1$,
\begin{equation*}
\begin{aligned}
\int_{0}^{\infty} &\left\lvert\left[(1+K^2)h^2\partial^2_{x_1} -2(1+iK\xi_{n-1})h\partial_{x_1} + (1-\lvert \xi \rvert^2)\right] \hat{v}(x_1,\xi)\right\rvert^2 \,dt \\
\simeq &\int_{0}^{\infty} \lvert (1-h^2\partial^2_{x_1} + \lvert\xi\rvert^2)  \hat{v}(x_1,\xi)\rvert^2 \,dt\\
&\simeq  \int_{0}^{\infty} \lvert \mathcal{F}((1-h^2\Delta_{x})v)(x_1,\xi)\rvert^2 \,dt
\simeq  \lVert v \rVert^2_{H^2(\R)}.
\end{aligned}
\end{equation*}
Hence taking $\mu$ and $\delta$ small enough, from \eqref{step_3} we get
\begin{equation*}
\begin{aligned}
\lVert \g \rVert^2_{L^2(\R)} &\leq \left(C_{\mu} + C_{\delta}\right)\lVert u\rVert^2_{L^2(\R)}
\leq \frac{1}{2}\lVert u\rVert^2_{L^2(\R)}.
\end{aligned}
\end{equation*}

Combining it with \eqref{step_4} and Lemma \ref{lemma for g} we see
\begin{equation*}
\begin{aligned}
\frac{h}{\sqrt{\e}} \lVert w_s \rVert_{L^2(\R)}
\leq &C \left\lVert J_s \tilde{\mathcal{L}}_{\phi_{\e}}(J_s^{-1}w_s) \right\rVert_{H^{-1}(\R)} + \mathcal{O}(h^{\infty})\lVert w \rVert_{L^2(\R)}\\
\leq &C \left\lVert \tilde{\mathcal{L}}_{\phi_{\e}}J_sJ_s^{-1}w_s \right\rVert_{H^{-1}(\R)} + hC\lVert J_s^{-1}w_s \rVert_{H^1(\R)} \\
&+ \mathcal{O}(h^{\infty})\lVert w \rVert_{L^2(\R)}\\
\leq &C \left\lVert \tilde{\mathcal{L}}_{\phi_{\e}}w_s \right\rVert_{H^{-1}(\R)} + hC\lVert w_s \rVert_{L^2(\R)} + \mathcal{O}(h^{\infty})\lVert w \rVert_{L^2(\R)}\\
\mbox{So, } \frac{h}{\sqrt{\e}} \lVert w_s \rVert_{L^2(\R)} \leq& C \left\lVert \tilde{\mathcal{L}}_{\phi_{\e}}w_s \right\rVert_{H^{-1}(\R)} +  \mathcal{O}(h^{\infty})\lVert w \rVert_{L^2(\R)}, \mbox{ for }\e>0\mbox{ small enogh}.
\end{aligned}
\end{equation*}

This completes the proof of Lemma: \eqref{small freq lemma}.
\qed

\subsection{Large frequency case (Lemma \ref{large freq lemma}):}
Following the approach of \cite{chung_2014} (section 5) let us consider the function
\begin{equation*}
F:\mathbb{R}^{n-1}\rightarrow\mathbb{C}, \qquad \text{defined as}
\end{equation*}
\begin{equation*}
\overline{F(\xi)} = \frac{1}{1+\lvert K \rvert^2}\left( 1+iK\xi_{n-1} + \sqrt{2iK\xi_{n-1} - (K\xi_{n-1})^2 + (1+\lvert K\rvert^2)\lvert \xi \rvert^2 - \lvert K \rvert^2 } \right),
\end{equation*}
but now consider the branch of square root so that it contains the non negative real axis.

Now, $F$ is smooth except where
\begin{equation*}
\tilde{F}(\xi) = 2iK\xi_{n-1} - (K\xi_{n-1})^2 + (1 + \lvert K \rvert^2)\lvert \xi \rvert^2 - \lvert K \rvert^2
\end{equation*}
belongs to the non positive real axis. That is when
\begin{equation*}
\xi_{n-1} = 0 \qquad and \quad \lvert \xi \rvert^2 \leq \frac{\lvert K \rvert^2}{1 + \lvert K \rvert^2},
\end{equation*}
So, on the support of $1-\rho(\xi)$, $F$ is smooth. Observe that now the real part of the square root is non negative, so, $\Re F$ and $\lvert F \rvert $ are bounded below by $\frac{1}{1 + K^2}$. Consider a smooth function $F_l(\xi)$ on $\mathbb{R}^{n-1}$ so that $F_l(\xi) = F(\xi)$ on support of $1-\rho(\xi)$ and $\Re F_l,\lvert F_l \rvert > \frac{1}{2 + 2K^2}$. For large $\xi$ we have $\Re F_l,\lvert F_l \rvert \simeq 1 + \lvert \xi\rvert$.

If, $\frac{K^2}{1+ K^2} < r_0 < r_1$ and $0<\delta_0 < \delta_1$, one can get $F_l = F$ and $F_l$ to be smooth for $\lvert \xi \rvert^2 \geq r_0$ and $\xi_{n-1} \geq \delta_0$.

For $u\in \mathcal{S}(\R)$ define the operator $J_l$ by
\begin{equation*}
\widehat{J_l u}(x_1,\xi) = \left( {F_l(\xi)} + h\partial_{x_1} \right)\hat{u}(x_1,\xi).
\end{equation*}
The adjoint of the above operator is defined as
\begin{equation*}
\widehat{J^*_l u}(x_1,\xi) = \left( {\overline{F_l(\xi)}} - h\partial_{x_1} \right)\hat{u}(x_1,\xi).
\end{equation*}
These operators have right inverses defined as
\begin{equation*}
\begin{aligned}
\widehat{J^{-1}_l}u(x_1,\xi) = h^{-1} \int_{0}^{x_1} \hat{u}(t,\xi)e^{\frac{t-x_1}{h}F_l(\xi)} \ dt,\\
\mbox{and}\quad
\widehat{{J^*_l}^{-1}}u(x_1,\xi) = h^{-1} \int_{x_1}^{\infty} \hat{u}(t,\xi)e^{\frac{x_1-t}{h}\overline{F_l(\xi)}} \ dt.
\end{aligned}
\end{equation*}
Each of the above is well defined functions in $\mathcal{S}(\R)$.

For the operators defined above, we have the following lemmas.
\begin{lemma}\label{boundedness lemma_1}
$J_l, J^*_l, J_l^{-1}, {J_l^*}^{-1}$ extends as bounded maps
\begin{equation*}
\begin{aligned}
J_l, J_l^* : H^1(\R) \rightarrow L^2(\R),
\quad \mbox{and} \quad
J_l^{-1}, {J_l^*}^{-1} : L^2(\R) \rightarrow H^1(\R).
\end{aligned}
\end{equation*}
Moreover, the extensions of $J^*_l$ and ${J^*_l}^{-1}$ are isomorphisms.
\end{lemma}

\begin{lemma}\label{commutativity with 2nd order op_1}
\textbf{(A.)} Let $w \in \mathcal{S}(\R)$, if $Q$ is a second order semiclassical differential operator with bounded and smooth coefficients, then
\begin{equation*}
\lVert (J_lQ - QJ_l)w \rVert_{H^{-1}(\R)} \leq hC_{\delta}\lVert w \rVert_{H^1(\R)}.
\end{equation*}
\textbf{(B.)} Let $\chi \in \mathcal{S}(\R)$, then
\begin{equation*}
\Vert J_l\chi J_l^{-1} w \rVert_{L^2(\R)} \geq \lVert \chi w \rVert_{L^2(\R)} - hC_{\delta}\lVert w\rVert_{L^2(\R)}.
\end{equation*}
\end{lemma}

\begin{lemma}\label{lemma for g_1}
For $u\in \mathcal{S}(\mathbb{R}^{n-1}_{1+})$, if $g$ is defined by
\begin{equation*}
\begin{aligned}
\hat{\g}(x_1,\xi) &= \frac{2\Re F_l(\xi)}{h}\int_{0}^{\infty} \hat{u}(t,\xi)e^{-\frac{1}{h}\left( t\overline{F_l(\xi)} + x_1F_l(\xi) \right)} dt\\
\mbox{then}\qquad
&\lVert J_l u \rVert_{H^{-1}(\R)} \simeq \lVert u - \g\rVert_{L^2(\R)}.
\end{aligned}
\end{equation*}
Here $\Re(z) \equiv \Re z$ means the real part of the complex number $z$.
\end{lemma}
The proofs of these lemmas are same as the proof of the corresponding lemmas for the small frequency operator $J_s$.

Now, by the Carleman estimate \eqref{initial carleman est on omega tilde} and the similar arguments used in the previous subsection, we have
\begin{equation*}
\frac{h}{\sqrt{\e}} \lVert w_l \rVert_{L^2(\R)} \leq C \lVert \tilde{\mathcal{L}}_{\phi_{\e}}(J_l^{-1}w_l) \rVert_{L^2(\R)} + \mathcal{O}(h^{\infty})\lVert w \rVert_{L^2(\R)}.
\end{equation*}
And now again we want to combine the above equation with Lemma \eqref{lemma for g_1} to get
\begin{equation*}
\frac{h}{\sqrt{\e}} \lVert w_l \rVert_{L^2(\R)} \leq C \lVert J_l \tilde{\mathcal{L}}_{\phi_{\e}}(J_l^{-1}w_l) \rVert_{H^{-1}(\R)} + \mathcal{O}(h^{\infty})\lVert w \rVert_{L^2(\R)}.
\end{equation*}
For this we need that the function $\g$, defined in Lemma \eqref{lemma for g_1} for $u = \tilde{\mathcal{L}}_{\phi_{\e}}(J_l^{-1}w_l)$, to satisfy the following estimate
\begin{equation*}
\lVert \g \rVert_{L^2(\R)} \leq \frac{1}{2}\lVert u\rVert_{L^2(\R)} + \mathcal{O}(h)\lVert w_l\rVert_{L^2(\R)}.
\end{equation*}
Here we cannot proceed by the arguments used in the last section for small frequency instead we factorize the operator $\tilde{\mathcal{L}}_{\phi_{\e}}$.
Let $\zeta(\xi)$ be a smooth cutoff function defined as
\begin{equation*}
\begin{aligned}
\zeta(\xi) =
	\begin{cases} 	1, 	& \mbox{if } \lvert \xi \rvert^2 \geq r_1 \mbox{ or } \xi_{n-1} \geq \delta_1\\
					0, 	& \mbox{if } \lvert \xi \rvert^2 \leq r_0 \mbox{ and } \xi_{n-1} \leq \delta_0.
	\end{cases}
\end{aligned}
\end{equation*}

Let $G_{s} = (1-\zeta(\xi))F_l(\xi)$ and consider the following symbol
\begin{equation*}
G_{\pm} = \zeta(\xi)\frac{\alpha + i\beta_{f}.\xi \pm \sqrt{(\alpha + i\beta_{f}.\xi)^2 - (1+\lvert\gamma_{f}\rvert^2)(\alpha^2 - L(\p{x},\xi))}}{1+ \lvert \gamma_f\rvert^2} + G_s(\xi),
\end{equation*}
where $L(\p{x},\xi)$ is the semiclassical symbol of the second order operator $\mathcal{L}$.

The branch of the square root is with non-negative real part. Note that $G_{\pm}$ is discontinuous if $(\alpha + i\beta_{f}.\xi)^2 - (1+\lvert\gamma_{f}\rvert^2)(\alpha^2 - \lvert\xi\rvert^2)$ lies on non-positive real axis, that is when
\begin{align}\label{step_5}
\langle \beta_f , \xi \rangle = 0  \qquad and \qquad L(\p{x},\xi)\leq \frac{\alpha^2\lvert\gamma_f\rvert^2}{1 + \lvert \gamma_f\rvert^2}.
\end{align}
Now, for $\mu$ small enough we have $\beta_f \simeq Ke_{n-1}$, $\lvert\gamma_f\rvert \simeq K$, for $\delta>0$ small enough $L(\p{x},\xi) \simeq \lvert \xi \rvert^2$ and for $h$ small enough we have $\alpha \simeq 1$, hence, \eqref{step_5} cannot happen on the support of $\zeta$. So, $G_{\pm}$ is smooth on the support of $\zeta$ and are symbols of order 1 on $\mathbb{R}^{n-1}$.

Denote $T_{a}$ as the operator corresponding to symbol $a$. Then
\begin{equation*}
\begin{aligned}
&(h\partial_{x_1} - T_{G_{+}})(1+ \lvert \gamma_f\rvert^2) (h\partial_{x_1} - T_{G_{-}}) \\
=& 	(1+ \lvert \gamma_f\rvert^2)h^2\partial_{x_1}^2
	-(\alpha + \beta_f.h\nabla_{\p{x}} )h\partial_{x_1}T_{\zeta}
	+(\alpha^2 + h^2\mathcal{L} )T_{\zeta^2}\\
&	+(1+ \lvert \gamma_f\rvert^2)T_{G_{s}}
	+(1+ \lvert \gamma_f\rvert^2)(T_{G_{+}G_{-}} + T_{G_{-}G_{+}} + T_{G_{s}G_{s}} )T_{G_{s}}
	+hE_1,
\end{aligned}
\end{equation*}
where $E_1$ is a semiclassical pseudodifferential operator consisting first order operators in $\mathbb{R}^{n-1}$ and $\partial_{x_1}$ which is bounded from $H^1(\R)$ to $L^2(\R)$.
Let us take $v= {J_l}^{-1}w_l$.
Observe that for fixed $x_1$, as a function of $\xi$, $\hat{w}_l(x_1,\xi)$ is supported in support of $(1-\rho(\xi))$. Therefore $T_{\zeta}v = v$ and since $\zeta \equiv 1$ in that set so, $T_{\zeta^2}v = v$. Moreover as $G_s \equiv 0$ on the support of $1-\rho(\xi)$, so, $T_{G_s}v = 0$. Hence,
\begin{equation*}
\begin{aligned}
&	(h\partial_{x_1} - T_{G_{+}})(1+ \lvert \gamma_f\rvert^2) (h\partial_{x_1} - T_{G_{-}})v\\
=&	(1+ \lvert \gamma_f\rvert^2)h^2\partial_{x_1}^2v
	-(\alpha + \langle\beta_f,h\nabla_{\p{x}}\rangle )h\partial_{x_1}v
	+(\alpha^2 + h^2\mathcal{L} )v
	+hE_1v	\\
=&	\tilde{\mathcal{L}}_{\phi_{\e}}v + hE_1v.
\end{aligned}
\end{equation*}

Let us write $z=(1+\lvert \gamma_f\rvert^2)(h\partial_{x_1} - T_{G_{-}} )v$, then $\tilde{\mathcal{L}}_{\phi_{\e}}v = (h\partial_{x_1} - T_{G_{+}})z - hE_1 v$.

Now, let us calculate $\hat{\g}(x_1,\xi)$ corresponding to $u = \tilde{\mathcal{L}}_{\phi_{\e}}(J_l^{-1}w_l)$.
\begin{equation*}
\begin{aligned}
\hat{\g}(x_1,\xi)
=& 	\frac{2\Re F_l(\xi)}{h}\int_{0}^{\infty}	\widehat{\tilde{\mathcal{L}}_{\phi_{\e}}v}(t,\xi)  e^{-\frac{1}{h}(t\overline{F_l(\xi)} + x_1F_l(\xi)} 	\ dt	\\
=& 	\frac{2\Re F_l(\xi)}{h}\int_{0}^{\infty}	\mathcal{F}\left([h\partial_t - T_{G_{+}}]z\right)(t,\xi)  e^{-\frac{1}{h}(t\overline{F_l(\xi)} + x_1F_l(\xi))} 	\ dt	\\
& 	-\frac{2\Re F_l(\xi)}{h}\int_{0}^{\infty}	h\widehat{E_1v}(t,\xi)  e^{-\frac{1}{h}(t\overline{F_l(\xi)} + x_1F_l(\xi))} 	\ dt	\\
=& 	\frac{2\Re F_l(\xi)}{h}\int_{0}^{\infty}	\mathcal{F}\left([T_{\overline{F_l}} - T_{G_{+}}]z\right)(t,\xi)  e^{-\frac{1}{h}(t\overline{F_l(\xi)} + x_1F_l(\xi))} 	\ dt	\\
& 	-\frac{2\Re F_l(\xi)}{h}\int_{0}^{\infty}	h\widehat{E_1v}(t,\xi)  e^{-\frac{1}{h}(t\overline{F_l(\xi)} + x_1F_l(\xi))} 	\ dt	\\
\end{aligned}
\end{equation*}

Hence, by the same calculation used in the proof of Lemma \eqref{lemma for g} we have
\begin{equation*}
\begin{aligned}
\lVert \g \rVert^2_{L^2(\R)} \leq  \left\lVert T_{\overline{F_l} - G_{+}} z \right\rVert^2_{L^2(\R)} + h^2\lVert E_1 v \rVert^2_{L^2(\R)}.
\end{aligned}
\end{equation*}

Now, to estimate $\left\lVert T_{\overline{F_l} - G_{+}} z \right\rVert^2_{L^2(\R)}$ let us calculate the term $(\overline{F_l} - G_{+})$.
We get
\begin{equation}\label{step_9}
\begin{aligned}
\left\lVert T_{\overline{F_l} - G_{+}}z\right\rVert^2_{L^2(\R)}
\leq C_{\mu,\delta}\lVert z \rVert^2_{H^1(\R)},
\end{aligned}
\end{equation}
for $h$, $\mu$ and $\delta$ small enough. Here the constant $C_{\mu,\delta}$ are small if $\mu,\delta$ is small.
And hence we get,
\begin{equation*}
\begin{aligned}
\lVert \g\rVert^2_{L^2(\R)}
&\leq  C_{\mu,\delta}\lVert z\rVert^2_{H^1(\R)} + h^2\lVert E_1v\rVert^2_{L^2(\R)}\\
&\leq  C_{\mu,\delta}\lVert z\rVert^2_{H^1(\R)} + h^2\lVert v\rVert^2_{H^1(\R)}.
\end{aligned}
\end{equation*}

As $\tilde{\mathcal{L}}_{\phi_{\e}} v = (h\partial_{x_1} - T_{G_{+}})z - hE_1v$, we get
\begin{equation*}
\begin{aligned}
		\lVert\tilde{\mathcal{L}}_{\phi_{\e}} v \rVert^2_{L^2(\R)}
\geq& 	\lVert(h\partial_{x_1} - T_{G_{+}})z \rVert^2_{L^2(\R)} - h^2\lVert E_1v\rVert^2_{L^2(\R)}\\
\geq& 	\lVert J_l^*z \rVert^2_{L^2(\R)} - \lVert T_{\overline{F_l} - G_{+}}z \rVert^2_{L^2(\R)} - h^2\lVert E_1v\rVert^2_{L^2(\R)}\\
\geq& 	C\left(\lVert z \rVert^2_{H^1(\R)} - C_{\mu,\delta}\lVert z \rVert^2_{H^1(\R)} - h^2\lVert v\rVert^2_{H^1(\R)}\right)\\
\geq&	C\lVert z \rVert^2_{H^1(\R)} - Ch^2\lVert v\rVert^2_{H^1(\R)}.
\end{aligned}
\end{equation*}

Hence,
\begin{equation*}
\begin{aligned}
		\lVert \g\rVert^2_{L^2(\R)}
\leq& 	C_{\mu,\delta} \lVert \tilde{\mathcal{L}}_{\phi_{\e}} v \rVert^2_{L^2(\R)}
		+ h^2\lVert v\rVert^2_{H^1(\R)}
\leq& \frac{1}{2}\lVert \tilde{\mathcal{L}}_{\phi_{\e}} v \rVert^2_{L^2(\R)} + h^2\lVert w_l\rVert^2_{L^2(\R)}.
\end{aligned}
\end{equation*}

Now using the Lemmas \eqref{boundedness lemma_1},\eqref{commutativity with 2nd order op_1} and \eqref{lemma for g_1} we get that
\begin{equation*}
\begin{aligned}
\frac{h}{\sqrt{\e}} \lVert w_l \rVert^2_{L^2(\R)} \leq C\lVert J_l \tilde{\mathcal{L}}_{\phi_{\e}} \chi_2 J_l^{-1}w_l \rVert_{H^{-1}(\R)} + \mathcal{O}(h^{\infty})\lVert w \rVert_{L^2(\R)} .
\end{aligned}
\end{equation*}
Using the same techniques in small frequency case, we can get
\begin{equation*}
\frac{h}{\sqrt{\e}} \lVert w_l \rVert^2_{L^2(\R)} \leq C\lVert \tilde{\mathcal{L}}_{\phi_{\e}} w_l \rVert_{H^{-1}(\R)} + \mathcal{O}(h^{\infty})\lVert w \rVert_{L^2(\R)} .
\end{equation*}
Hence, the proof of Lemma \eqref{large freq lemma} is complete.
\qed

\subsection{Estimate on admissible manifolds}
In the last subsection we finished the proof of Proposition \ref{Target_Proposition}. Using the corollary of the same proposition we had Estimate \eqref{Target_estimate}.
Now our aim is to prove the Carleman estimate \eqref{Int_Car_Est} using the Estimate \eqref{Target_estimate}.

Recall our assumptions:
\begin{enumerate}
\item $M \subset \mathbb{R}^n$
\item The metric $\lvert g^{jk}-\delta^{jk} \rvert < \delta$.
\item $E$ can be parameterized by a smooth real valued function $f$ so that
\begin{equation*}
E \subset \{ (x_1,\p{x}): x_1=f(\p{x}), \p{x} \in \mathbb{R}^{n-1} \}.
\end{equation*}
\item The function $f$ is such that $\lvert \nabla_{g_0} f(\p{x}) - Ke_{n-1}\rvert \leq \mu$ for some small $\mu$.
\end{enumerate}
First to remove assumption 4 observe that one can cover $\Omega$ by finitely many open sets $U_{j}$ such that for some appropriate coordinate system in each $U_j$ there is $K_j$ so that
\begin{equation*}
\lvert \nabla_{g_0} f(\p{x}) - K_je_{n-1}\rvert_{g_0} \leq \frac{\mu}{2},\quad \mbox{for }\mu>0\mbox{ small}.
\end{equation*}
Now, there exist $\delta_0>0$ such that if $0<\delta<\delta_0$, then
$\lvert \nabla_{g_0} f(\p{x}) - K_je_{n-1}\rvert \leq \mu$.

Let $\chi_1, \chi_2, \dots \chi_m$ be a partition of unity subordinate to the cover $U_1, U_2, \dots U_m$. For $w \in C^{\infty}_{c}({M})$,
\begin{equation*}
w = \sum_{i=1}^{m}\chi_iw = \sum_{i=1}^{m}w_i,
\end{equation*}
where $w_i \in C^{\infty}_{c}({M}\cap U_i)$. Apply \eqref{Target_estimate} in ${M}\cap U_i$ for $w_i$,
\begin{equation*}
\frac{h}{\sqrt{\e}}\lVert w_j \rVert_{L^2({M}\cap U_i)}
\leq C\lVert {\mathcal{L}}_{\phi_{\e}} w_i \rVert_{H^{-1}(A_0)},\quad \forall w\in C^{\infty}_c({M} \cap U_i).
\end{equation*}

Then for $w\in C_c^{\infty}({M})$,
\begin{equation*}
\begin{aligned}
\frac{h}{\sqrt{\e}}\lVert w \rVert_{L^2({M})}
&	\leq \sum_{i=1}^{m} \frac{h}{\sqrt{\e}}\lVert w_i \rVert_{L^2({M}\cap U_i)}
\leq C\sum_{i=1}^{m}\lVert {\mathcal{L}}_{\phi_{\e}} \chi_i w \rVert_{H^{-1}(A_0)}\\
&\leq C\sum_{i=1}^{m}\lVert \chi_i{\mathcal{L}}_{\phi_{\e}} w \rVert_{H^{-1}(A_0)}
		+Ch\sum_{i=1}^{m}\lVert w \rVert_{L^2(M)}\\
\mbox{So,}\quad \frac{h}{\sqrt{\e}}\lVert w \rVert_{L^2({M})}
&	\leq C\lVert {\mathcal{L}}_{\phi_{\e}} w \rVert_{H^{-1}(A_0)}.
\end{aligned}
\end{equation*}

Hence, we have for $w\in C_c^{\infty}({M})$ and for any smooth function $f$
\begin{equation*}
\frac{h}{\sqrt{\e}}\lVert w \rVert_{L^2({M})}
\leq C\lVert {\mathcal{L}}_{\phi_{\e}} w \rVert_{H^{-1}(A_0)}.
\end{equation*}

Now, to remove assumption 3, let $M$ is covered by open sets $V_1,V_2,\dots,V_m$ such that each $V_j \cap E$ can be viewed as a graph of a smooth function $f_j$. Let $A_j$ denotes the set $\{(x_1,\p{x}) : x_1 \geq f_j(\p{x}) \}$ containing $\Omega$. Consider a cutoff function $\chi_j \in C^{\infty}(A_j)$ so that $\chi_j \equiv 1$ on $V_j$ and $\chi_j \equiv 0$ on $A_j\setminus \Omega$.
Hence for $w \in C^{\infty}_c(M)$, we get
\begin{equation*}
\begin{aligned}
&\frac{h}{\sqrt{\e}}\lVert \chi_jw \rVert_{L^2(M)}
\leq C\lVert {\mathcal{L}}_{\phi_{\e}} \chi_jw \rVert_{H^{-1}(A_j)}\\
\implies \quad
&\frac{h}{\sqrt{\e}}\lVert w \rVert_{L^2(M)}
\leq C\sum_{i=1}^{m}\lVert \chi_i{\mathcal{L}}_{\phi_{\e}} w \rVert_{H^{-1}(A_j)}
\end{aligned}
\end{equation*}

Now multiplying by $\chi_j$ is a bounded linear operator from $H^1_0(A_j)$ to $H^1_0(\Omega)$. Hence by duality we get it to be a bounded linear map from $H^{-1}(\Omega)$ to $H^{-1}(A_j)$.
Hence, for small enough $0<h<\e$,
\begin{equation}\label{step_0}
\frac{h}{\sqrt{\e}}\lVert w \rVert_{L^2(M)}
\leq C\lVert {\mathcal{L}}_{\phi_{\e}} w \rVert_{H^{-1}(\Omega)}, \quad \forall w \in C^{\infty}_c(M).
\end{equation}
We have the above estimate true for any compact domain $M\subset \Omega \subset \mathbb{R}^n$ with $E \subset (\partial M \setminus \Gamma_{D}) \cap \partial \Omega$.

Now, to deal with assumption 2 we notice that one can find a coordinate chart near any point $\p{p} \in M_0$ so that the metric $g_0^{jk}$ is $\delta^{jk}$ at $p_0$. Hence we can find a neighborhood $\p{U}_{\p{p}}$ of $\p{p}$ and a coordinate map $\p{\psi}_{\p{p}}$ on $\p{U}_{\p{p}}$ so that in this coordinate chart $\lvert g_0^{jk} - \delta^{jk} \rvert < \delta$.
Define the map $\psi_p(x_1,\p{x}) = (x_1,\p{\psi}_{\p{x}})$ on $U_p = (\mathbb{R} \times \p{U}_{\p{p}}) \cap \Omega $.

Using the fact that $\Omega$ is compact subset of $\mathbb{R}_{+} \times M_0$ we will find a finite number partition of unity $(U_i,\chi_i)_{i=1,\dots,n}$ on $\Omega$ such that on each $U_i$ we have $\lvert g^{jk} - \delta^{jk}\rvert \leq \delta$.
Hence, using the Estimate \eqref{step_0} we get
\begin{equation*}
\frac{h}{\sqrt{\e}}\lVert \chi_iu \rVert_{L^2(U_i)}
\leq C\lVert {\mathcal{L}}_{\phi_{\e}} (\chi_i u) \rVert_{H^{-1}(\Omega)}.
\end{equation*}

Note that here the operator $\mathcal{L}_{\phi_{\e}}$ remains unchanged as $\Delta_g$ is defined in a coordinate independent way and the function $\phi_{\e}$ depends on $x_1$ variable only. Hence, for $i=1,\dots,n$ we have
\begin{equation*}
\frac{h}{\sqrt{\e}}\lVert \chi_iu \rVert_{L^2(U_i)}
\leq C\lVert {\mathcal{L}}_{\phi_{\e}} (\chi_i u) \rVert_{H^{-1}(\Omega)}.
\end{equation*}

Hence for $w \in C^{\infty}_c(M)$,
\begin{equation}\label{step_7}
\begin{aligned}
\frac{h}{\sqrt{\e}}\lVert u \rVert_{L^2(M)}
\leq \frac{h}{\sqrt{\e}}\sum_{i=1}^{n}\lVert \chi_iu \rVert_{L^2(U_i)}
\leq &C\sum_{i=1}^{n} \lVert {\mathcal{L}}_{\phi_{\e}} (\chi_i u) \rVert_{H^{-1}(\Omega)}\\
\leq &C\sum_{i=1}^{n} \left(\lVert{\chi_i\mathcal{L}}_{\phi_{\e}} u \rVert_{H^{-1}(\Omega)}
			+ ch\lVert u \rVert_{L^{2}(M)} \right)\\
\end{aligned}
\end{equation}

Hence for $\e$ and $h$ small enough and $u \in C^{\infty}_c(M)$ we get
\begin{equation}\label{step_8}
\frac{h}{\sqrt{\e}}\lVert u \rVert_{L^2(M)}
\leq \lVert{\mathcal{L}}_{\phi_{\e}} u \rVert_{H^{-1}(\Omega)}.
\end{equation}

Till now we have worked on $M \subset \Omega \subset \mathbb{R}^n$ and in the last estimate above we have successfully removed the assumption 2.

Now removing assumption 1 follows from the fact that $M_0$ is simple. Hence, there is a diffeomorphism $\psi_1 : M_0 \to \overline{B(0,1)} \subset \mathbb{R}^{n-1}$, where $B(0,1) = \{\p{x} \in \mathbb{R}^{n-1} : \lvert \p{x} \rvert < 1 \}$. Now define the map
\begin{equation*}
\psi : \mathbb{R}_{+} \times M_0 \to \R, \qquad \mbox{defined as }\quad \psi(x_1,\p{x}) := (x_1, \psi_1(\p{x})).
\end{equation*}

Let $u\circ \psi^{-1} \in C^{\infty}_{c}(\psi(M))$.
Observe that a calculation shows that
\begin{equation*}
\lVert u \rVert_{L^2(M)} \simeq \lVert u\circ \psi_p^{-1} \rVert_{L^2(\psi(M))}, \quad \mbox{as well as} \quad \lVert u \rVert_{H^1(\Omega)} \simeq \lVert u \circ \psi^{-1} \rVert_{H^1(\psi(\Omega))}.
\end{equation*}
Using duality one can prove $\lVert u \rVert_{H^{-1}(\Omega)} \simeq \lVert u \circ \psi^{-1} \rVert_{H^{-1}(\psi(\Omega))}$.

Let us denote the coordinates in $\psi(\Omega)$ as $y=(y_1,\p{y})$ and the elements in $\Omega$ are $x = (x_1,\p{x}) \in \mathbb{R} \times M_0$. Having this notation, for $u\in C^{\infty}_{c}(M)$, we get
\begin{equation*}
\begin{aligned}
\mathcal{L}_{\phi_{\e}} u(x)
=& e^{\phi_{\e}/h}h^2\Delta_ge^{-\phi_{\e}/h} u(\psi^{-1}(y))\\
=& e^{\phi_{\e}/h}h^2\Delta_{\tilde{g}}e^{-\phi_{\e}/h} \left(u \circ \psi^{-1}\right)(y)
+ hE_1\left(u \circ \psi^{-1}\right)(y),
\end{aligned}
\end{equation*}
where $\tilde{g}^{jk}(y) = \frac{\partial \psi^{-1}_j}{\partial y_{\alpha}}(y) g^{\alpha\beta}(\psi^{-1}(y))\frac{\partial \psi^{-1}_k}{\partial y_{\beta}}(y)$ is the metric defined on $\psi(\Omega)$ and $E_1$ is a first order semiclassical differential operator on $\psi(\Omega)$. From \eqref{step_8} we have
\begin{equation*}
\frac{h}{\sqrt{\e}} \lVert v \rVert_{L^2(\psi(M))} \leq C \lVert e^{\phi_{\e}/h}h^2\Delta_{\tilde{g}}e^{-\phi_{\e}/h} v \rVert_{H^{-1}(\psi(\Omega))}, \quad \forall v \in C_{c}^{\infty}(\psi(M)).
\end{equation*}

Writing $v=u\circ\psi^{-1}$ and using triangle inequality along with the equivalence of norms under the change of coordinate $\psi$, we get
\begin{equation*}
\begin{aligned}
&\frac{h}{\sqrt{\e}} \lVert u \rVert_{L^{2}(M)}
\leq C \lVert \mathcal{L}_{\phi_{\e}} u \rVert_{H^{-1}(\Omega)}
+ h\lVert E_1 u \rVert_{H^{-1}(\psi(\Omega))},
\quad \forall u \in C_{c}^{\infty}(M)\\
&\implies
\frac{h}{\sqrt{\e}} \lVert u \rVert_{L^{2}(M)}
\leq C \lVert \mathcal{L}_{\phi_{\e}} u \rVert_{H^{-1}(\Omega)}
+ h\lVert u \rVert_{L^2(\Omega)},
\quad \forall u \in C_{c}^{\infty}(M)\\
&\implies
\frac{h}{\sqrt{\e}} \lVert u \rVert_{L^{2}(M)}
\leq C \lVert \mathcal{L}_{\phi_{\e}} u \rVert_{H^{-1}(\Omega)}, \quad \mbox{for small }\e>0, \quad \mbox{and } \forall u \in C_{c}^{\infty}(M).
\end{aligned}
\end{equation*}

Hence we get the following estimate
\begin{equation}\label{step_1}
\frac{h}{\sqrt{\e}} \lVert u \rVert_{L^{2}(M)}
\leq C \lVert \mathcal{L}_{\phi_{\e}} u \rVert_{H^{-1}(\Omega)}, \quad \mbox{for small }\e>0, \quad \mbox{and } \forall u \in C_{c}^{\infty}(M).
\end{equation}
Hence we have successfully removed the assumptions 1, 2, 3 and 4 stated above.
Now following the same steps in the proof of the boundary Carleman estimate to add the lower order terms in the right hand side and replace $\phi_{\e}$ by $\phi$, we get
\begin{equation}\label{final interior carleman estimate}
\begin{aligned}
h\lVert u \rVert_{L^2(M)}
\leq C\lVert {\mathcal{L}}_{\phi,B,q} u \rVert_{H^{-1}(\Omega)}, \quad \forall u \in C^{\infty}_c(M).
\end{aligned}
\end{equation}

\section{Construction of the solution}\label{CGO}
In this section, we construct suitable solutions of the  operator $\mathcal{L}_{\phi,B,q}$ .
But at first we state the following proposition which can be proved by using Hahn-Banach extension and Riesz representation theorem. We will skip the proof here, for details see \cite[Proposition 4.4]{F_K_S_U}.

\begin{proposition}\label{existance of r}
Let $\phi(x)=x_1$ and $\mathcal{L}_{\phi,B,q}$ is as before.
For any $v \in L^2(M)$, there exists $u\in H^1(M)$ such that
\begin{equation*}
\begin{aligned}
\mathcal{L}^*_{\phi,B,q} u &= v \quad on \quad M\\
u|_{E} &= 0
\end{aligned}
\end{equation*}
\end{proposition}
\begin{equation*}
\mbox{and} \qquad h\lVert u\rVert_{H^1(M)} \leq C\lVert v \rVert_{L^2(M)}.
\end{equation*}
Here $\mathcal{L}^*_{\phi,B,q}$ is the $L^2$ adjoint of $\mathcal{L}_{\phi,B,q}$.

\begin{remark}\label{existance}
Note that as the Carleman estimate \eqref{final interior carleman estimate} is true for $\phi(x) = x_1$ and the operator $\mathcal{L}^*_{\phi,B,q}$ is of the form
\begin{equation*}
\mathcal{L}^*_{\phi,B,q} = e^{-\phi/h}\mathcal{L}^*_{B,q}e^{\phi/h}.
\end{equation*}
Hence the above proposition guarantees a solution $u$ of the equation
\begin{equation*}
\begin{aligned}
\mathcal{L}_{B,q}e^{x_1/h} u &= 0 \quad \mbox{in } M\\
u|_{E} & = 0.
\end{aligned}
\end{equation*}
\end{remark}

Recall that we have assumed that $M_0$ is simple, hence, we can extend $M_0$ slightly so that the extended domain $\tilde{M}_0$ is also simple. Let $W$ be a point outside $M_0$. As the manifold is simple one can consider a global geodesic normal coordinate system $(r,w)$ on $\tilde{M}_0$ center at $W$.
We pose the form of of the solution $u$ by
\begin{equation}
u = e^{\frac{1}{h}(x_1 + ir)}(v_s(x) + r_1(x)) - e^{l/h}b(x).
\end{equation}
Let us now study the conditions on $v_s, r_1, l, b$ so that $u$ satisfies the conditions
\begin{equation}\label{mag_shrd_eqn}
\begin{aligned}
\mathcal{L}_{\B{1},\q{1}} u &= 0\\
u|_{E} &= 0.
\end{aligned}
\end{equation}
Putting the solution $u$ in the equation $\mathcal{L}_{\B{1},\q{1}}u = 0$ and writing $\rho = -(x_1 + ir)$, we get
\begin{equation}\label{eqn_r_1}
\begin{aligned}
-e^{\rho/h}h^2\mathcal{L}_{\B{1},\q{1}}e^{-\rho/h}r_1(x) =& \left[ -\lvert \nabla \rho \rvert^2 + h\left(2\langle \nabla \rho, \nabla.\rangle + \Delta\rho + 2i\langle d\rho,\B{1}\rangle \right) \right] v_s(x) \\
&- e^{\rho/h}\mathcal{L}_{\B{1},\q{1}}e^{l/h}b(x) +h^2\mathcal{L}_{\B{1},\q{1}}v_s(x)
\end{aligned}
\end{equation}
Now, we will choose $v_s, l, b$ so that the right hand side of the above equation is bounded in $L^2(M)$.
Observe that we already have $\langle \nabla \rho, \nabla \rho \rangle = 0$.
Next we seek for a solution $v_s$ satisfying
\begin{equation*}
2\langle \nabla \rho, \nabla v_s\rangle  + \Delta\rho v_s + 2i\langle d\rho,\B{1}\rangle v_s = 0
\end{equation*}

A calculation shows
\begin{equation*}
-\Delta_g \rho = \lvert g \rvert^{-1/2}\partial_{x_1}\left(\frac{\lvert g\rvert^{1/2}}{c} \right) + \lvert g \rvert^{-1/2}\partial_{r}\left(\frac{\lvert g\rvert^{1/2}}{c}i \right) = \frac{1}{c}\bar{\partial} log \frac{\lvert g\rvert}{c^2},
\end{equation*}
where $\bar{\partial} = \frac{1}{2} \left(\partial_{x_1} + i\partial_r \right)$.
A solution of the above transport equation is given by
\begin{equation*}
v_s = \lvert g \rvert^{-1/4}c^{1/2}e^{i\Phi(r)}b_1(x_1,r)b_2(w)
\end{equation*}
where $\overline{\partial}{\Phi}(r) + \frac{1}{2}\left( \B{1}_1 + i\B{1}_r \right)= 0$, $\overline{\partial}b_1(x_1,r) = 0$. The function $b_2(w)$ is smooth and supported in $\{w \in \mathbb{R}^{n-2}: \lvert w \rvert <\delta \}$ for $\delta>0$ small. Here $\B{1}_r$ denotes the coefficient of $\B{1}$ in the direction of $r$ .

Now,
\begin{equation*}
\begin{aligned}
e^{-\rho/h}F & = \mathcal{L}_{\B{1},\q{1}}e^{l/h}b(x)\\
			 & = \frac{1}{h^2}\left[ \langle \nabla l, \nabla l\rangle_g  \right]e^{l/h} + \frac{1}{h}\left[ 2\langle \nabla l, \nabla b \rangle_g + 2i\langle \B{1},dl \rangle_g b + (\Delta_g l)b \right]e^{l/h} \\
			 & \quad + \left[ (\Delta_g b) + 2i\langle \B{1}, db \rangle_g + \tilde{q}_1 b \right]e^{l/h}.
\end{aligned}
\end{equation*}
In order to make the support condition true, we seek for a solution to the equation
\begin{equation*}
\begin{aligned}
\langle \nabla l, \nabla l \rangle_g &= 0 \quad in \quad M\\
l|_{E} &= -\rho.
\end{aligned}
\end{equation*}
In order to avoid duplicating the solution $l=-\rho$ we impose the condition
\begin{equation*}
\partial_{\nu}l|_{E} = \partial_\nu(\rho)|_{E}.
\end{equation*}
To construct a solution, pick boundary normal coordinates $(t,\xi)$ near $E$ such that $t$ are coordinates along $E$ and $\xi$ is perpendicular to $E$. Note that this $\xi$ is different from the $\xi$ defined in the previous sections as frequency variable in Carleman estimates.

Suppose in a small neighborhood of a point on $E$, $l$ takes the form of a power series
\begin{equation*}
l(t,\xi) = \sum_{j=0}^{\infty}a_j(t)\xi^j.
\end{equation*}
Then, in boundary normal coordinates
\begin{equation*}
\begin{aligned}
\nabla l 	= (\nabla_t l, \partial_\xi l)
			= \left( \sum_{j=0}^{\infty}\nabla_t a_j(t)\xi^j, \sum_{j=0}^{\infty} ja_j(t)\xi^{j-1} \right)
\end{aligned}
\end{equation*}
Hence, $\langle \nabla l, \nabla l \rangle_g = 0$ implies
\begin{equation*}
\begin{aligned}
0 =& 	\sum_{j+k=2} jka_ja_k
		+ \xi\left( \sum_{j+k=3} jka_ja_k \right)
		+ \xi^2\left( \sum_{j+k=3} jka_ja_k \right)+ \dots\\
&	+ \sum_{j+k=0} \langle \nabla_t a_j, \nabla_t a_k \rangle_{g_t}
				+ \xi\left( \sum_{j+k=1} \langle \nabla_t a_j, \nabla_t a_k\rangle_{g_t} \right)\\
	&	+ \xi^2\left( \sum_{j+k=2} \langle \nabla_t a_j, \nabla_t a_k \rangle_{g_t} \right) + \dots
\end{aligned}
\end{equation*}
Considering each power of $\xi$ separately gives a sequence of equations
\begin{equation}\label{eqn_a_j}
\sum_{j+k = m} \langle \nabla_t a_j, \nabla_t a_k \rangle_{g_t} + \sum_{j+k = m+2} jk a_j a_k = 0, \qquad \text{for each}\quad m=0,1,2,3, \dots.
\end{equation}
One can determine $a_0$ and $a_1$ from the boundary conditions. If $m\geq 1$, and all $a_j$ are known for $j\leq m$, the only unknown part of \eqref{eqn_a_j} is $2(m+1)a_1a_{m+1}$. Note that $a_1 = \partial_\nu \rho$, and recall that $E \subset \{x\in \partial_{-}M : \Re (\partial_\nu \rho) > {\delta} \}$. So, we can divide the equation involving $a_1a_{m+1}$ by $a_1$ and solve for $a_{m+1}$.

Observe that there is no guarantee that this power series of $l$ will converge outside $\xi = 0$. However, we will construct a smooth function for which the Taylor series coincides with the above power series at $\xi = 0$. For this purpose let us consider $\chi : \mathbb{R} \rightarrow [0,1]$ be a smooth cutoff function which is supported in $[0,1]$ and $\chi \equiv 1$ in $[-\frac{1}{2}, \frac{1}{2}]$.
Define $p_j = \max\limits_{k \leq j} \{ \lVert a_k(t)\rVert_{C^{k}(\Omega)} , 1 \}$.
Now construct $l$ as
\begin{equation*}
l = \sum_{j=0}^{\infty} a_j(t) \chi(p_j \xi) \xi^j = \sum_{j=0}^{\infty} c_j(t) \xi^j.
\end{equation*}
This defines a smooth function on $\Omega$ whose Taylor series at $\xi = 0$ is same as the power series calculated earlier. By the same calculation as above, it shows that the coefficient of $\xi^m$ in $\langle \nabla l, \nabla l \rangle_g$ is
\begin{equation*}
\sum_{j+k = m} \langle \nabla_t c_j, \nabla_t c_k\rangle_{g_t} + \sum_{j+k = m}\left( \partial_{\xi}c_j + (j+1)c_{j+1} \right)\left( \partial_{\xi}c_k + (k+1)c_{k+1} \right).
\end{equation*}
Now observe that, by the above construction $c_j = a_j$ for $j\leq k$ in the region $\xi \leq \frac{1}{p_k}$. As $a_j$ are solutions of the above equations, so, in the region $\xi \leq \frac{1}{p_k}$ we have $\langle \nabla l, \nabla l \rangle_g = \mathcal{O}(\xi^k)$. Now we can do it for any $p$, which makes $\langle \nabla l, \nabla l \rangle_g = \mathcal{O}(d(x,E)^{\infty})$.

As $\partial_{\nu} \Re(\rho)>0$ on $E$ and $\Re(l) |_{E} = -\Re(\rho)|_{E}$ with $\partial_{\nu} \Re(l) = \partial_{\nu} \Re(\rho)$ we can say that in a neighborhood of $E$
\begin{equation}
\Re\ l(x) = -\Re\ \rho(x) - k(x),
\end{equation}
where $k(x) \simeq (d(x,E))$.

Similarly we can construct $b$ as an approximate solution for the equation
\begin{equation*}
\begin{aligned}
-i\langle\nabla l , \nabla b\rangle + \langle \nabla l, \B{1}b \rangle &= \mathcal{O}(d(x,E)^{\infty})\\
b|_{E} &= v_s|_{E}.
\end{aligned}
\end{equation*}
Observe that if $b$ solves the above equation, then, we can multiply a smooth cutoff function with $b$ and will still solve the above equation. Hence, we can assume that $b$ is supported in a very small neighborhood of $E$. Then
\begin{equation*}
\begin{aligned}
\lvert h^2 \mathcal{L}_{\B{1},\q{1}} (e^{l/h}b) \rvert &= \lvert e^{l/h}\rvert \left( \mathcal{O}(d(x,E)^\infty) + \mathcal{O}(h^2) \right)\\
&= e^{-\Re(\rho/h)}e^{-k/h}\left( \mathcal{O}(d(x,E)^\infty) + \mathcal{O}(h^2) \right).
\end{aligned}
\end{equation*}
If $d(x,E) \leq h^{1/2}$, for $h$ small, then $\lvert h^2 \mathcal{L}_{\B{1},\q{1}} (e^{l/h}b) \rvert$ is of order $e^{\phi/h}\mathcal{O}(h^2)$. On the other hand if $d(x,E) \geq h^{1/2}$ the term is again of order $e^{\phi/h}\mathcal{O}(h^2)$ due to $e^{-k/h}$.

Therefore, $e^{-\rho/h}v_s - e^{l/h}b = 0$ on $E$.
Denote
\begin{equation*}
f = e^{\rho/h}h^2\mathcal{L}_{\B{1},\q{1}} (e^{-\rho/h}v_s - e^{l/h}b),
\end{equation*}
where $\lVert f \rVert_{L^2(M)} = \mathcal{O}(h^2)$.
Thanks to Proposition \ref{existance of r} and Remark \ref{existance} we get a solution $\p{r}$ of the following equation:
\begin{equation}\label{eqn_r}
\begin{aligned}
\left(e^{-x_1/h}\mathcal{L}_{\B{1},\q{1}}e^{x_1/h}\p{r}_1\right) &= e^{ir/h}f, \qquad \mbox{in }\quad M\\
\p{r}_1|_{E} = 0.
\end{aligned}
\end{equation}
Taking $r_1 = e^{-ir/h}\p{r}_1$, we have produced a solution $u$, of the equation \eqref{mag_shrd_eqn} having the form
\begin{equation*}
u = e^{-\rho/h}(v_s(x) + r_1) - e^{l/h}b \quad with \quad supp(u|_{\partial M}) \subset \Gamma_{D}
\end{equation*}
and $\lVert r_1 \rVert_{L^2(M)}  = \mathcal{O}(h)$ as $h \rightarrow 0$.
Similarly for the Carleman weight $\phi(x) = -x_1$, using a standard technique as in \cite[Proposition 4.4]{F_K_S_U} we obtain a solution $v$ of the equation $(\mathcal{L}_{\B{2},\q{2}})^* v = 0$ in $M$, having the form
\begin{equation*}
v(x) = e^{-(x_1-ir)/h}(\tilde{v}_s(x) + r_2)
\end{equation*}
where $\lVert r_2 \rVert_{H^1(M)}  = \mathcal{O}(h)$ as $h \rightarrow 0$ and
\begin{align*}
\tilde{v}_s = \lvert g \rvert^{-1/4}c^{1/2}e^{i\tilde{\Phi}(x)}b_3(x_1,r)b_4(w)
\end{align*}
where $\partial\tilde{\Phi}(x) + \frac{1}{2}\left( \overline{\B{2}}_1 - i\overline{\B{2}}_r \right)= 0$, ${\partial}b_3(x_1,r) = 0$ and $b_4(w)$ is smooth and supported in a $B(0,\delta)$ for $\delta>0$ small. Here, $\partial = 1/2\left( \partial_{x_1} - i\partial_r \right)$.

\section{Integral identity} \label{int_id}
Let us now consider $\B{j}$, $\q{j}$ as in the Theorem \ref{main_theorem}.
Let $u_j$ be solution of the equation $\mathcal{L}_{\B{j},\q{j}} = 0$, $(j=1,2)$ in $M$ and have the same boundary data with $supp(u_j|_{\partial M})\subset\Gamma_{D}$.
Let $v$ be a solution of the equation $(\mathcal{L}_{\B{2},\q{2}})^*v = 0$ in $M$.
Recall that $(\mathcal{L}_{\B{2},\q{2}})^*$ denotes $L^2$ adjoint of $\mathcal{L}_{\B{2},\q{2}}$. Using $\mathcal{C}_{\B{1},\q{1}}^{\Gamma_D,\Gamma_N} = \mathcal{C}_{\B{1},\q{1}}^{\Gamma_D,\Gamma_N}$ we get the following proposition:
\begin{proposition}\label{integral_id_proposition}
If $\Gamma_D, \Gamma_N \subset \partial M$ and assume that $\phi,u_1,u_2,v$ and $\mathcal{L}_{\B{j},\q{j}}$ are as above, then
\begin{equation}\label{integral_identity}
\begin{aligned}
\int_{M} &\left[ 2i\langle (\B{1} - \B{2}), du_1\rangle - (\tilde{q}^{(1)} - \tilde{q}^{(2)})u_1 \right] \overline{v} \,dV \\
&= -i\int_{\partial M\setminus\Gamma_{N}} d_{\B{2}}[u_1-u_2](\nu) \overline{v}\,dS,
\end{aligned}
\end{equation}
where $\tilde{q}^{(l)} = \q{l} - i[\lvert g \rvert^{-1/2}\partial_{x_k}(\lvert g \rvert^{1/2}g^{jk}\B{l}_j)] + \lvert \B{l} \rvert^2_g$ and $l=1,2$.
\end{proposition}

\begin{proof}
Observe that a straight forward calculation implies
\begin{equation*}
(\mathcal{L}_{\B{2},\q{2}})^*w = \lvert g \rvert^{-1/2}(\partial_{x_j} -  i\overline{\B{2}_j}) \lvert g \rvert^{1/2}{g^{jk}}(\partial_{x_k} -  i\overline{\B{2}_k})w + \overline{\q{2}}w = \overline{\mathcal{L}_{\B{2},\q{2}}}w
\end{equation*}
where $()^*$ denotes the $L^2$ dual.
Then for any $u,w \in H^1(M)$ one has
\begin{equation*}
\begin{aligned}
&\langle \mathcal{L}_{\B{2},\q{2}}u,w\rangle_{L^2(M)} - \langle u,(\mathcal{L}_{\B{2},\q{2}})^*w\rangle_{L^2(M)} \\
	   =&-i\int_{\partial M} d_{\B{2}}u(\nu)\overline{w}\ dS
	     +i\int_{\partial M} u (\overline{d_{\overline{\B{2}}}\ w(\nu)})\ dS.
\end{aligned}
\end{equation*}
Let us now take $u=(u_1-u_2)$ and $w = v$, then by our assumptions $u|_{\partial M} = 0$. Putting this into the above identity we get
\begin{equation*}
\begin{aligned}
\langle \mathcal{L}_{\B{2},\q{2}}u,v\rangle_{L^2(M)} - \langle u,(\mathcal{L}_{\B{2},\q{2}})^*v\rangle_{L^2(M)}
	   =-i\int_{\partial M} d_{\B{2}}(u_1-u_2)(\nu)\overline{v}\ dS.
\end{aligned}
\end{equation*}
Now, as $(\mathcal{L}_{\B{2},\q{2}})^*v = 0$ and $\mathcal{L}_{\B{j},\q{j}}u_j = 0$, we get
\begin{equation*}
\begin{aligned}
2i\int_{M} \langle(\B{1} - \B{2}),du_1\rangle_g\overline{v}\ dV + \int_{M} (\tilde{q}^{(2)} - \tilde{q}^{(1)})u_1\overline{v} = -i\int_{\partial M\setminus\Gamma_{N}} du(\nu) \overline{v}\ dS.
\end{aligned}
\end{equation*}
\end{proof}

Now we estimate the right hand side integral in \eqref{integral_identity}. Using the boundary Carleman estimate \eqref{bdy_Carleman_estimate} we will show that the right hand side integral goes to 0 as $h \rightarrow 0$.
We have for $u=u_1-u_2$ and $\phi = -x_1$,
\begin{equation*}
\begin{aligned}
&\left\lvert \int_{\partial M\setminus\Gamma_{N}} d_{\B{2}}u(\nu) \overline{v} \D S \right\rvert^2\\
\leq &	C\langle e^{\phi/h}\partial_{\nu} u, e^{\phi/h}\partial_{\nu} u \rangle_{\partial M\setminus\Gamma_{N}} \\
\leq &	C\left\langle \lvert \partial_{\nu}\phi\rvert \partial_{\nu} (e^{\phi/h}u), \partial_{\nu} (e^{\phi/h}u) \right\rangle_{ \{ \partial_{\nu} \phi \leq -\delta_1 \} }, \qquad \mbox{using \eqref{bdy_Carleman_estimate} and } \delta_1>0\mbox{ small}\\
\leq& Ch\lVert e^{\phi/h}\mathcal{L}_{\B{2},\q{2}} u \rVert^2_{L^2(M)} + \left\langle \lvert \partial_{\nu}\phi\rvert \partial_{\nu} (e^{\phi/h}u), \partial_{\nu} (e^{\phi/h}u) \right\rangle_{ \{ \partial_{\nu} \phi \geq 0 \} }.
\end{aligned}
\end{equation*}

Now, as
$\B{1} = \B{2}$ on $\partial M$, we get $\partial_{\nu}(e^{\phi/h}u)|_{ \{ \partial_{\nu}\geq 0 \} } \equiv 0$.
Hence,
\begin{equation*}
\begin{aligned}
\left\lvert \int_{\partial M\setminus\Gamma_{N}} d_{\B{2}}u(\nu) \overline{v} \D S \right\rvert
\leq& Ch^{1/2}\lVert e^{\phi/h} \left( \mathcal{L}_{\B{2},\q{2}} - \mathcal{L}_{\B{1},\q{1}} \right) u_1 \rVert_{L^2(M)}\\
\leq& Ch^{1/2}\left\lVert \left\langle e^{\phi/h}\left( \B{1} - \B{2} \right), du_1 \right\rangle_g \right\rVert_{L^2(M)}\\
&+ Ch^{1/2}\lVert e^{\phi/h}(\q{1} - \q{2})u_1 \rVert_{L^2(M)}.
\end{aligned}
\end{equation*}
\begin{equation*}
\begin{gathered}
\mbox{Note that}
\qquad \lim\limits_{h \to 0} h\left\lVert \langle e^{\phi/h}\left( \B{1} - \B{2} \right), du_1 \rangle_g \right\rVert_{L^2(M)} = 0\\
\mbox{and}\quad
\lim_{h \to 0} h^{1/2}\lVert e^{\phi/h}(\q{1} - \q{2})u_1 \rVert_{L^2(M)} = 0.
\end{gathered}
\end{equation*}
\begin{equation}\label{bdy etimate_1}
\mbox{So,}\qquad
\lim\limits_{h\to 0} \left\lvert h\int_{\partial M\setminus\Gamma_{N}} d_{\B{2}}u(\nu) \overline{v} \D S \right\rvert = 0.
\end{equation}

Moreover if we assume $\B{1} = \B{2}$ in $M$, then
\begin{equation}\label{bdy etimate}
\lim\limits_{h\to 0} \left\lvert \int_{\partial M\setminus\Gamma_{N}} d_{\B{2}}u(\nu) \overline{v} \D S \right\rvert = 0.
\end{equation}

\section{Recovering the lower order perturbations} \label{agrt}
Here we follow the method in \cite{F_K_S_U} to recover the lower order perturbations. Note that the Fourier transforms given in this section are in the classical sense.

\subsection{The magnetic field}
Let $u$ be a solution of $\mathcal{L}_{\B{1},\q{1}}u = 0$ of the form
\begin{equation*}
u(x_1,r,w) = e^{\frac{1}{h}(x_1 + ir)}(v_s(\p{x})+r_1(x)) + \p{u},
\end{equation*}
with $\p{u} = e^{l/h}b$ and $supp(u|_{\partial M}) \subset \Gamma_{D}$.

Here, $v_s(\p{x}) = \lvert g\rvert^{-1/4}c^{1/2}e^{i\Phi} b_1(x_1,r)b_2(w)$ where
$\overline{\partial}\Phi(x) + \frac{1}{2}\left[{\B{1}_1} + i{\B{1}_r} \right]= 0$, $\overline{\partial}{b}_1(x_1,r) = 0$ and $b_2(w)$ is smooth and supported in a $B(0,\delta)$ for $\delta>0$ small and $b_2(0) = 1$.

Let $u_2$ solve
\begin{equation*}
\begin{aligned}
\mathcal{L}_{\B{2},\q{2}} u_2 &= 0 \quad in \quad M\\
u_2|_{\Gamma_D} &= u|_{\Gamma_D}. 
\end{aligned}
\end{equation*}
Then, by our assumption of Theorem \ref{main_theorem} we have $d_{\B{2}}u_2(\nu)|_{\Gamma_{N}} = d_{\B{1}}u(\nu)|_{\Gamma_{N}}$.

Let $v \in H^1(M)$ be a solution of $(\mathcal{L}_{\B{2},\q{2}})^* v = 0$ of the form
\begin{equation*}
v(x_1,r,w) = e^{-\bar{\rho}/h}(\tilde{v}_s(\p{x})+r_2(x)),
\end{equation*}
with $supp(v|_{\partial M}) \subset \Gamma_{N}$.
Similarly,
\begin{equation*}
\tilde{v}_s(\p{x}) = \lvert g \rvert^{-1/4}c^{1/2}e^{\tilde{\Phi}}b_3(x_1,r)b_4(w)
\end{equation*}
with $\partial\tilde{\Phi}(x) + 2\left[ \overline{\B{2}_1} - i\overline{\B{2}_r} \right]= 0$, $\bar{\partial}b_3(x_1,r) = 0$ and $b_4(w)$ is smooth and supported in a $B(0,\delta)$ for $\delta>0$ small and $b_4(0) = 1$.

Writing $B=(\B{1} - \B{2})$ in Proposition \ref{integral_id_proposition} we get
\begin{equation*}
\begin{aligned}
&\int_{\partial M\setminus\Gamma_{N}} d(u-u_2)(\nu) \overline{v}\ dS\\
=&\int_{M} \left[2i\langle B, du\rangle\overline{v} - (\tilde{q}_1 - \tilde{q}_2)u\overline{v}  \right]\ dV\\
=& 	-2i\frac{1}{h}\int_M (B_1 + iB_r) \{v_s\overline{\tilde{v}_s} + r_1\overline{\tilde{v}_s}+\overline{r_2}v_s + r_1\overline{r_2} \}\ dV \\
&	-2i\int_M \langle B,d v_s\rangle  \{ \overline{\tilde{v_s}} + \overline{r_2} \}\ dV
	-2i\int_M \langle B,dr_1(x)\rangle_g (\overline{\tilde{v}_s}+\overline{r_2})\ dV
\end{aligned}
\end{equation*}\begin{equation}\label{full integral id}
\begin{aligned}&	-2i\int_M \langle B,d\p{u}(x)\rangle_g \overline{v(x)} dV\\
&	-\int_M (\tilde{q}^{(1)} - \tilde{q}^{(2)}) \{v_s(x)\overline{\tilde{v}_s(x)} + r_1\overline{\tilde{v}_s(x)}+\overline{r_2}v_s + r_1\overline{r_2} + \p{u}(x)\overline{v(x)}\}\ dV.
\end{aligned}
\end{equation}
First let us consider the 4th term on the right hand side of the last equality:
\begin{equation*}
\begin{aligned}
&\left\lvert \int_M \langle B,d\p{u}(x)\rangle_g \overline{v(x)}\ dV \right\rvert\\
&\leq \left\lvert \int_M \langle B,[h^{-1}bdl + db]\rangle_ge^{l/h} e^{-\rho/h}(\overline{v_s(x)} + \overline{r_2(x)})\ dV \right\rvert \\
&\leq \int_M e^{-k/h}\left\lvert \left(\frac{1}{h}\langle B,dl\rangle_gb + \langle B,db\rangle_g\right)(\overline{v_s(x)} + \overline{r_2(x)})\right\rvert\ dV\\
&\leq C (\lVert \nabla l \rVert_{L^2(M)} + \lVert \nabla b \rVert_{L^2(M)})(\lVert v_s \rVert_{L^2(M)} + \lVert r_2 \rVert_{L^2(M)}) \simeq \mathcal{O}(1).
\end{aligned}
\end{equation*}
Also observe that
\begin{equation*}
\begin{aligned}
\left\lvert \int_M (\tilde{q}^{(1)} - \tilde{q}^{(2)}) \p{u}(x)\overline{v(x)}\ dV\right\rvert
&\leq  	 \int_M e^{-\frac{k}{h}}\left\lvert(\tilde{q}^{(1)} - \tilde{q}^{(2)})b(x) \{\overline{v_s(x)} + \overline{r_2(x)} \}\right\rvert\ dV \\
&\simeq \mathcal{O}(h).
\end{aligned}
\end{equation*}
Now, multiplying the above identity \eqref{full integral id} by $h$ and taking $h \rightarrow 0$, we see that \eqref{bdy etimate} implies
\begin{equation*}
\int_M (B_1+iB_r) v_s\overline{\tilde{v_s}} \ dV = 0
\end{equation*}
Using the same analysis as in \cite[Subsection 6.2]{F_K_S_U} one can show that the above integral identity implies $d\B{1}=d\B{2}$ on $M$.
By simply connectedness of the domain $M$, $d(\B{1}-\B{2}) = 0$ implies
\begin{equation}\label{recovery of mag term}
(B^{(1)} - B^{(2)}) = dp, \quad \mbox{for some } p\in C^\infty(M). 
\end{equation}
Using the assumption $\B{1}=\B{2}$ on $\partial M$ we have $dp = 0$ on $\partial M$.
Hence subtracting a suitable constant we get $p = 0$ on $\partial M$.

\subsection{The potential}
Using $(B^{(1)} - B^{(2)}) = dp$ with $p\in C^\infty(M)$, $p|_{\partial M} = 0$ the gauge invariance allows us to assume that $\B{1} = \B{2}$ on $M$. Using the identity \eqref{full integral id} and replacing $\left(\B{1} - \B{2}\right) = 0$ we get
\begin{equation*}
\begin{aligned}
\int_M (\tilde{q}_1 - \tilde{q}_2) \{v_s(\p{x})\overline{\tilde{v}_s(\p{x})} + r_1\overline{\tilde{v}_s(\p{x})}+\overline{r_2}v_s + r_1\overline{r_2} \}\ dV \\
= \int_{\partial M\setminus\Gamma_{N}} d(u-u_2)(\nu) \overline{v}\ dS +\int_M (\tilde{q}_1 - \tilde{q}_2) \p{u}(x)\overline{v(x)}\ dV
\end{aligned}
\end{equation*}
\begin{equation*}
\mbox{where}\qquad
\tilde{q}_j = \q{j} - i[\lvert g \rvert^{-1/2}\partial_{x_l}(\lvert g \rvert^{1/2}g^{kl}\B{j}_k)] + \lvert \B{j} \rvert^2_g, \quad for \quad j = 1,2.
\end{equation*}
Now, assuming $\B{1} = \B{2}$ on $M$ we get $\tilde{q}_1 - \tilde{q}_2 = \q{1} - \q{2}$ on $M$.
Hence, using the estimate \eqref{bdy etimate} and the technique used in previous section to show the last integral is of order $h$, we get
\begin{equation*}
\begin{aligned}
\lim\limits_{h \rightarrow 0} &\int_M (\tilde{q}_1 - \tilde{q}_2) \{v_s(\p{x})\overline{\tilde{v}_s(\p{x})} + r_1\overline{\tilde{v}_s(\p{x})}+\overline{r_2}v_s + r_1\overline{r_2} \}\ dV = 0\\
\implies \quad &\int_M (\q{1} - \q{2})v_s(\p{x})\overline{\tilde{v}_s(\p{x})}\ dV = 0.
\end{aligned}
\end{equation*}
Using the same calculation given in \cite[Subsection 6.1]{F_K_S_U} we get $\q{1}=\q{2}$ in $M$.

\section*{Acknowledgments} The research was supported by the Doctoral fellowship of TIFR CAM, Bangalore, India and the Airbus Group Corporate Foundation Chair in ``Mathematics of Complex Systems" established at TIFR CAM and TIFR ICTS, Bangalore, India.
The author would like to thank the reviewer for making crucial comments and valuable suggestions which helped the author immensely. Finally the author would like to thank Venkateswaran P. Krishnan for his support and guidance throughout.



\medskip
Received September 2016; revised January 2018.
\medskip

{\it E-mail address: }arkatifr@gmail.com\\


\begin{thebibliography}{100}

\bibitem{Astala} (MR2195135) [10.4007/annals.2006.163.265]
\newblock K. Astala and L. P{\"a}iv{\"a}rinta,
\newblock \doititle{Calder\'on's inverse conductivity problem in the plane},
\newblock \emph{Ann. of Math. (2)}, \textbf{163} (2006), 265--299.

\bibitem{BUK} (MR2387648)
\newblock A. L. Bukhgeim,
\newblock Recovering a {p}otential from {c}auchy data in the two-dimensional case,
\newblock \emph{J. Inverse Ill-Posed Probl.}, \textbf{16} (2008), 19--33.

\bibitem{B_U} (MR1900557) [10.1081/PDE-120002868]
\newblock A. L. Bukhgeim and G. Uhlmann,
\newblock \doititle{Recovering a potential from partial cauchy data},
\newblock \emph{Communications in Partial Differential Equations}, \textbf{27} (2002), 653--668.

\bibitem{Calderon_Paper} (MR590275)
\newblock A. P. Calder{\'o}n,
\newblock On an inverse boundary value problem,
\newblock In \emph{Seminar on {N}umerical {A}nalysis and its {A}pplications to {C}ontinuum {P}hysics ({R}io de {J}aneiro, 1980)}, pages 65--73. Soc. Brasil. Mat., Rio de Janeiro, 1980.

\bibitem{chung_3} (MR3295954) [10.3934/ipi.2014.8.959]
\newblock F. J. Chung,
\newblock \doititle{Partial data for the {n}eumann-{d}irichlet magnetic {s}chr\"odinger inverse problem},
\newblock \emph{Inverse Probl. Imaging}, \textbf{8} (2014), 959--989.

\bibitem{chung_2014} (MR3219502) [10.2140/apde.2014.7.117]
\newblock F. J. Chung,
\newblock \doititle{A partial data result for the magnetic {s}chr\"odinger inverse problem},
\newblock \emph{Anal. PDE}, \textbf{7} (2014), 117--157.

\bibitem{chung_2} (MR3345369) [10.1007/s00041-014-9379-5]
\newblock F. J. Chung,
\newblock \doititle{Partial data for the {n}eumann-to-{d}irichlet map},
\newblock \emph{J. Fourier Anal. Appl.}, \textbf{21} (2015), 628--665.

\bibitem{Chung2013} (MR3611013) [10.2140/apde.2017.10.43]
\newblock F. J. Chung, M. Salo and L. Tzou,
\newblock \doititle{Partial data inverse problems for the hodge laplacian},
\newblock \emph{Anal. PDE,} \textbf{10} (2017), 43--93, \url{https://arxiv.org/abs/1310.4616}

\bibitem{F_K_S_U} (MR2534094) [10.1007/s00222-009-0196-4]
\newblock D. D. S. Ferreira, C. E. Kenig, M. Salo and G. Uhlmann,
\newblock \doititle{Limiting {c}arleman weights and anisotropic inverse problems},
\newblock \emph{Invent. Math.}, \textbf{178} (2009), 119--171.

\bibitem{DOS} (MR2287913) [10.1007/s00220-006-0151-9]
\newblock D. D. S. Ferreira, C. E. Kenig, J. Sj{\"o}strand and G. Uhlmann,
\newblock \doititle{Determining a magnetic {s}chr\"odinger operator from partial {c}auchy data},
\newblock \emph{Comm. Math. Phys.}, \textbf{271} (2007), 467--488.

\bibitem{HOR} (MR1275197) [10.5802/aif.1371]
\newblock L. H\"{o}rmander,
\newblock \doititle{Remarks on Holmgren's uniqueness theorem},
\newblock \emph{Ann. Inst. Fourier}, \textbf{43} (1993), 1223--1251.

\bibitem{Iman3} (MR2770947) [10.1073/pnas.1011681107]
\newblock O. Y. Imanuvilov, G. Uhlmann and M. Yamamoto,
\newblock \doititle{Determination of second-order elliptic operators in two dimensions from partial {C}auchy data},
\newblock \emph{Proc. Natl. Acad. Sci. USA}, \textbf{108} (2011), 467--472.

\bibitem{Iman1} (MR2629983) [10.1090/S0894-0347-10-00656-9]
\newblock O. Y. Imanuvilov, G. Uhlmann and M. Yamamoto,
\newblock \doititle{The {C}alder\'on problem with partial data in two dimensions},
\newblock \emph{J. Amer. Math. Soc.}, \textbf{23} (2010), 655--691.

\bibitem{Iman2} (MR2819949)
\newblock O. Y. Imanuvilov, G. Uhlmann and M. Yamamoto,
\newblock Inverse boundary value problem by measuring {D}irichlet data and {N}eumann data on disjoint sets,
\newblock \emph{Inverse Problems}, \textbf{27} (2011), 085007, 26pp.

\bibitem{ISA} (MR2262748) [10.3934/ipi.2007.1.95]
\newblock V. Isakov,
\newblock \doititle{On uniqueness in the inverse conductivity problem with local data},
\newblock \emph{Inverse Probl. Imaging}, \textbf{1} (2007), 95--105.

\bibitem{K_S} (MR3198591) [10.2140/apde.2013.6.2003]
\newblock C. E. Kenig and M. Salo,
\newblock \doititle{The {C}alder\'on problem with partial data on manifolds and applications},
\newblock \emph{Anal. PDE}, \textbf{6} (2013), 2003--2048.

\bibitem{KEN} (MR2299741) [10.4007/annals.2007.165.567]
\newblock C. E. Kenig, J. Sj{\"o}strand and G. Uhlmann,
\newblock \doititle{The {C}alder\'on problem with partial data},
\newblock \emph{Ann. of Math. (2)}, \textbf{165} (2007), 567--591.

\bibitem{KNU} (MR2282273) [10.3934/ipi.2007.1.349]
\newblock K. Knudsen and M. Salo,
\newblock \doititle{Determining nonsmooth first order terms from partial boundary measurements},
\newblock \emph{Inverse Probl. Imaging}, \textbf{1} (2007), 349--369.

\bibitem{KIM2} (MR2209749) [10.1080/03605300500361610]
\newblock K. Knudsen,
\newblock \doititle{The {C}alder\'on problem with partial data for less smooth conductivities},
\newblock \emph{Comm. Partial Differential Equations}, \textbf{31} (2006), 57--71.

\bibitem{Krupchyk2017}
\newblock K. Krupchyk and G. Uhlmann,
\newblock Inverse problems for magnetic schr\"{o}dinger operators in transversally anisotropic geometries,
\newblock {ArXiv \url{https://arxiv.org/abs/1702.07974}}

\bibitem{NA} (MR1370758) [10.2307/2118653]
\newblock A. I. Nachman,
\newblock \doititle{Global uniqueness for a two-dimensional inverse boundary value problem},
\newblock \emph{Ann. of Math. (2)}, \textbf{143} (1996), 71--96.

\bibitem{S_N_U} (MR1354996) [10.1007/BF01460996]
\newblock G. Nakamura, Z. Sun and G. Uhlmann,
\newblock \doititle{Global identifiability for an inverse problem for the {S}chr\"odinger equation in a magnetic field},
\newblock \emph{Math. Ann.}, \textbf{303} (1995), 377--388.

\bibitem{SAL_2006} (MR2273968) [10.1080/03605300500530420]
\newblock M. Salo,
\newblock \doititle{Semiclassical pseudodifferential calculus and the reconstruction of a magnetic field},
\newblock \emph{Comm. Partial Differential Equations}, \textbf{31} (2006), 1639--1666.

\bibitem{SAL} (MR2481057) [10.1007/s00208-008-0301-9]
\newblock M. Salo and L. Tzou,
\newblock \doititle{Carleman estimates and inverse problems for {d}irac operators},
\newblock \emph{Math. Ann.}, \textbf{344} (2009), 161--184.

\bibitem{VS} (MR1374572)
\newblock V. A. Sharafutdinov,
\newblock \emph{Integral Geometry of Tensor Fields},
\newblock Inverse and Ill-posed Problems Series. VSP, Utrecht, Utrecht, the Netherlands, 1994.

\bibitem{SUN} (MR1179400)
\newblock Z. Sun,
\newblock An inverse boundary value problem for {S}chr\"odinger operators with vector potentials,
\newblock \emph{Trans. Amer. Math. Soc.}, \textbf{338} (1993), 953--969.

\bibitem{SYL} (MR873380) [10.2307/1971291]
\newblock J. Sylvester and G. Uhlmann,
\newblock \doititle{A global uniqueness theorem for an inverse boundary value problem},
\newblock \emph{Ann. of Math. (2)}, \textbf{125} (1987), 153--169.

\bibitem{TOL} (MR1617178) [10.1137/S0036141096301038]
\newblock C. F. Tolmasky,
\newblock \doititle{Exponentially growing solutions for nonsmooth first-order perturbations of the {L}aplacian},
\newblock \emph{SIAM J. Math. Anal.}, \textbf{29} (1998), 116--133 (electronic).

\end{thebibliography}
\end{document}